\documentclass{amsart}
\usepackage{amssymb}
\usepackage{amscd}
\usepackage{amsfonts}
\usepackage{graphicx}
\usepackage{latexsym}
\usepackage[all]{xy}
\usepackage{verbatim}
\usepackage{color}

\newtheorem{theorem}{Theorem}[section]
\newtheorem{lemma}[theorem]{Lemma}
\newtheorem{proposition}[theorem]{Proposition}
\newtheorem{corollary}[theorem]{Corollary}
\theoremstyle{definition}
\newtheorem{definition}[theorem]{Definition}
\newtheorem{example}[theorem]{Example}
\newtheorem{conjecture}[theorem]{Conjecture}

\newtheorem{remark}[theorem]{Remark}

\newcommand{\id}{\text{id}}

\newcommand{\sVec}{\operatorname{\operatorname{\mathsf{sVec}}}}

\renewcommand{\Vec}{\operatorname{\operatorname{\mathsf{Vec}}}}

\DeclareMathOperator{\Pic}{\operatorname{\mathsf{Pic}}}
\DeclareMathOperator{\BrPic}{\operatorname{\mathsf{BrPic}}}
\newcommand{\GL}{\text{GL}}
\newcommand{\Gr}{\text{Gr}}

\newcommand{\ad}{\text{ad}}
\DeclareMathOperator{\Aut}{\operatorname{\mathsf{Aut}}}
\DeclareMathOperator{\Sym}{\operatorname{\mathsf{Sym}}}
\DeclareMathOperator{\Ext}{\operatorname{\mathsf{Ext}}}

\DeclareMathOperator{\St}{\operatorname{\mathsf{St}}}
\DeclareMathOperator{\Ind}{\operatorname{\mathsf{Ind}}}
\DeclareMathOperator{\Rep}{\operatorname{\mathsf{Rep}}}
\DeclareMathOperator{\Inv}{\operatorname{\mathsf{Inv}}}
\DeclareMathOperator{\Hom}{\operatorname{\mathsf{Hom}}}

\newcommand{\op}{\text{op}}

\newcommand{\eps}{\varepsilon}

\newcommand{\B}{\mathcal{B}}
\newcommand{\C}{\mathcal{C}}
\newcommand{\D}{\mathcal{D}}
\newcommand{\E}{\mathcal{E}}

\newcommand{\Z}{\mathcal{Z}}

\renewcommand{\L}{\mathcal{L}}
\newcommand{\M}{\mathcal{M}}
\newcommand{\A}{\mathcal{A}}
\newcommand{\N}{\mathcal{N}}

\newcommand{\be}{\mathbf{1}}

\newcommand{\0}{_{(0)}}
\renewcommand{\1}{_{(1)}}
\renewcommand{\2}{_{(2)}}
\newcommand{\3}{_{(3)}}

\renewcommand{\be}{\mathbf{1}}

\newcommand{\bt}{\boxtimes}
\newcommand{\ot}{\otimes}

\newcommand{\beq}{\begin{equation}}
\newcommand{\eeq}{\end{equation}}

\newcommand{\bpf}{\begin{proof}}
\newcommand{\epf}{\end{proof}}

\newcommand{\bth}{\begin{theorem}}
\renewcommand{\eth}{\end{theorem}}
\newcommand{\bpr}{\begin{proposition}}
\newcommand{\epr}{\end{proposition}}
\newcommand{\ble}{\begin{lemma}}
\newcommand{\ele}{\end{lemma}}
\newcommand{\bco}{\begin{corollary}}
\newcommand{\eco}{\end{corollary}}
\newcommand{\bde}{\begin{definition}}
\newcommand{\ede}{\end{definition}}
\newcommand{\bex}{\begin{example}}
\newcommand{\eex}{\end{example}}
\newcommand{\bre}{\begin{remark}}
\newcommand{\ere}{\end{remark}}
\newcommand{\bcj}{\begin{conjecture}}
\newcommand{\ecj}{\end{conjecture}}

\newcommand{\End}{\text{End}}

\newcommand{\Hinv}{\text{H}_\text{inv}^2}

\begin{document}

\title{On the Brauer-Picard group of  a finite symmetric tensor category}
\author{Costel-Gabriel Bontea}
\address{Department of Mathematics and Statistics,
University of New Hampshire,  Durham, NH 03824, USA}
\email{costel.bontea@gmail.com}
\author{Dmitri Nikshych}
\address{Department of Mathematics and Statistics,
University of New Hampshire,  Durham, NH 03824, USA}
\email{nikshych@math.unh.edu}

\begin{abstract}
Let $\C_n$ denote the representation category of finite supergroup
$\wedge k^n \rtimes \mathbb{Z}/2\mathbb{Z}$. We compute the Brauer-Picard group  $\BrPic(\C_n)$  of $\C_n$.
This is done by identifying $\BrPic(\C_n)$ with the group of braided tensor autoequivalences
of the Drinfeld center of $\C_n$ and studying the action of the latter group on the categorical Lagrangian  Grassmannian of $\C_n$.
We show that this action corresponds to the action of a projective symplectic group  on a classical
Lagrangian Grassmannian.
\end{abstract}

\maketitle

\bigskip

\section{Introduction}

Let $\C$ be a finite tensor category. The notion of tensor product of $\C$-bimodule categories
\cite{ENO} generalizes the notion of tensor product of bimodules over a ring.
Bimodule categories invertible with respect to the tensor  product are of
particular interest.  Equivalence classes of  such categories  form a group, called the {\em Brauer-Picard group}
of $\C$ and denoted $\BrPic(\C)$. This group  plays a crucial role in construction and classification
of group-graded extensions of $\C$ \cite{ENO}.  When $\C$ is the category of finite dimensional representations of  a finite dimensional
Hopf algebra $H$ the group $\BrPic(\C)$ is known as the {\em full Brauer group} of $H$,
i.e., the {\em Brauer group} of the Drinfeld double of $H$, see \cite{COZ}. Brauer groups of Hopf algebras were
extensively studied in the literature, see, e.g.,  \cite{COZ, C, CC, CZ} among many other works.

Let $\Z(\C)$ denote the {\em Drinfeld center} of $\C$. There is a canonical isomorphism
\begin{equation}
\label{crucial iso}
\BrPic(\C) \cong  \Aut^{br}(\Z(\C)),
\end{equation}
where  $\Aut^{br}(\Z(\C))$ is the group of braided tensor  autoequivalences of $\Z(\C)$. This fact
was first established in \cite{ENO} for fusion categories and was later extended to finite tensor categories in \cite{DN}.

In practice it is much easier to work with the group  $\Aut^{br}(\Z(\C))$  than with $\BrPic(\C)$, since the multiplication
of  the latter is defined by an abstract universal property while for the former it is simply the composition
of functors. In addition, $\Aut^{br}(\Z(\C))$ can be viewed as a generalization
of  the classical orthogonal group which brings important geometric insights.
In particular, one can study actions of $\Aut^{br}(\Z(\C))$ on categorical analogues
of Grassmannians. This approach was used in \cite{NR} to compute  the Brauer-Picard
groups of pointed fusion categories.

In this paper we use the above approach in non-semisimple setting to
compute the Brauer-Picard group of the representation category $\C_n$ of finite supergroup
$E(n):=\wedge k^n \rtimes \mathbb{Z}/2\mathbb{Z}$ (note that $\C_n$ is the most general example of a symmetric
finite tensor category without non-trivial Tannakian subcategories \cite{De}). Namely, we show that
$ \BrPic(\C_n) \cong PSp_{2n}(k) \times \mathbb{Z}/2\mathbb{Z}$, where  $PSp_{2n}(k)$
denotes the projective symplectic group (Theorem~\ref{main thm} and Corollary~\ref{main cor}).

Let us explain details  of our argument. We study a projective representation $\rho$ of $\Aut^{br}(\Z(\C_n))$ on the
$2n$-dimensional  symplectic space $\Ext_{\Z(\C_n)}^{1}(\chi, \varepsilon)$,  where $\varepsilon$ and $\chi$
are the two one-dimensional representations of $\Z(\C_n)$ (this $\Ext$ space is identified with a space
of skew primitive elements in the Drinfeld double of $E(n)$, see Remark~\ref{skew-prims}).
The Lagrangian Grassmannian of this symplectic space
turns out to be equivariantly isomorphic to the  categorical Lagrangian Grassmannian $\mathbb{L}_0(\C_n)$,
i.e., to the set  of symmetric tensor subcategories of $\Z(\C_n)$ equivalent to~$\C_n$.
The stabilizer of a point of $\mathbb{L}_0(\C_n)$  is known (see Proposition~\ref{stabilizer of C} and
\cite[Proposition 6.8]{NR}).  Combining this analysis with the results of Carnovale and Cuadra \cite{CC}
about the Picard group of $\C_n$ we find the kernel of $\rho$  and  show that its  image
consists of classes  of symplectic matrices.  This allows to compute $\Aut^{br}(\Z(\C_n))$.

We  note that  the infinitesimal version of our result
(which says that the Lie algebra of outer twisted derivations of the double of $E(n)$
is isomorphic to the symplectic Lie algebra  $\mathfrak{s}\mathfrak{p}(2n)$)
was obtained by Cuadra and Davydov in \cite{CDa}.

The paper is organized as follows.

Section 2 contains preliminary material about Hopf algebras, tensor categories, and their Drinfeld centers.
Here we recall definition of the Brauer-Picard group of a tensor category.
We  also describe the projective action of the group of tensor autoequivalences of a tensor category
on certain $\Ext$-spaces.

In Section 3 we recall isomorphism \eqref{crucial iso} from \cite{ENO, DN} and discuss how braided
autoequivalences of the Drinfeld center of $\Rep(H)$, where $H$ is a finite dimensional
Hopf algebra, can be induced from invariant twists of $H$ and from invariant $2$-cocycles on $H$.
We also consider the action of the group of braided autoequivalences of the center on the
categorical Lagrangian Grassmannian.

In Section 4  we recall the structure of Hopf algebra $E(n)=\wedge k^n \rtimes \mathbb{Z}/2\mathbb{Z}$
and explicit description of its invariant twists  and $2$-cocycles following Bichon and Carnovale \cite{BC}.

In Section 5 we identify the categorical Lagrangian  Grassmannian of $\C_n$ with the classical
Lagrangian Grassmannian of  a symplectic form.

In Section 6 we compute  $\Aut^{br}(\Z(\C_n))$ using its projective representation on the $\Ext$-space.

\textbf{Acknowledgments.}
We are grateful to Juan Cuadra, Alexei Davydov,  Pavel Etingof, and Bojana Femic
for helpful discussions and valuable comments.
The work of the second named author  was partially supported  by the NSA grant H98230-13-1-0236.

\section{Preliminaries}

\subsection{General conventions}

We work over  an algebraically closed field $k$ of characteristic $0$.
Recall that a $k$-linear abelian category $\A$ is {\em finite}  \cite{EO} if
 \begin{enumerate}
\item[(i)] $\A$ has finite dimensional spaces of morphisms;
\item[(ii)] every object of $\A$ has finite length;
\item[(iii)] $\A$ has {\em enough projectives}, i.e., every simple object of $\A$ has a projective cover; and
\item[(iv)] there are finitely many isomorphism classes of simple objects in $\A$.
\end{enumerate}
All categories considered in this paper will be  finite $k$-linear abelian.
Any such category  is equivalent to the category $\Rep(A)$ of finite dimensional representations of
a finite dimensional  $k$-algebra $A$.  All functors between such categories will be additive and $k$-linear.

In this paper we freely use basic results of the theory of finite
tensor categories and module categories over them \cite{EGNO},
the theory of braided  categories \cite{JS, DGNO, EGNO}, and the theory
of Hopf algebras \cite{M}.

For a Hopf algebra $H$ we denote by $\Delta,\, S,\,\eps$,  the
comultiplication, antipode, and counit of $H$, respectively. We
use Sweedler's notation for comultiplication, writing $\Delta(x)=
x\1 \ot x\2,\, x\in H$. We will write $\Delta^\op$ for the
opposite comultiplication, i.e., $\Delta^\op(x) =x\2\ot x\1$. All
Hopf algebras considered in this paper are assumed to be finite
dimensional.

For a positive integer $n$ we denote by $\Sym_n(k)$ the additive
group of symmetric bilinear forms on $k^n$. When it is convenient,  we will identify
$\Sym_n(k)$ with the group of  symmetric $n$-by-$n$ matrices. We will denote
by $I_n$ the identity $n$-by-$n$ matrix.

\subsection{Tensor categories and Hopf algebras}

By a {\em tensor category} we mean a  finite rigid tensor category
whose  unit object $\be$ is simple.


\begin{example} \label{Hopf algebra}
For  a finite dimensional Hopf algebra $H$
the  category $\Rep(H)$ of finite dimensional representations of
$H$ is a finite tensor category.  In general,  a tensor category
$\A$ is equivalent to the representation category of some Hopf
algebra  if and only if there exists a  fiber functor (i.e., an exact faithful
tensor functor) $F: \A \to \Vec$, where $\Vec$ is the tensor category
of $k$-vector spaces.
\end{example}

By a tensor {\em subcategory} $\A$ of a tensor category $\B$ we mean the
image of  a fully faithful tensor
functor (i.e., embedding) $\iota: \A \hookrightarrow \B$. When no confusion is possible,
we will simply write $\A \subset \B$ to denote a tensor subcategory.

\begin{proposition} \label{subcategories of Rep(H)}
Let $H$ be a finite-dimensional Hopf algebra. The set of  tensor
subcategories of $\Rep(H)$ is in bijection
with the set of classes of surjective Hopf
algebra homomorphisms $p: H \to K$ under the following equivalence relation:
two surjective homomorphisms $p:H\to K$ and $p':H\to K'$ are equivalent if there is a Hopf algebra isomorphism
$f:K\xrightarrow{\sim} K'$ such that $f\circ p =p'$.
\end{proposition}
\begin{proof}
Any surjective homomorphism $p: H \to K$ induces an embedding
\begin{equation}
\iota(p) : \Rep(K)\to \Rep(H): V \mapsto V_H,
\end{equation}
where $V_H=V$ and  $hv :=p(h)v$ for all $h\in H$ and $v\in V$.
The image of $\iota(p)$ consists of isomorphism classes of
representations of $H$ that factor through $p$
and so it does not change when $p$ is  composed with an isomorphism.

Conversely, let $\iota:\A \to \Rep(H)$ be a tensor embedding.  Let $F: \Rep(H)\to \Vec$
be the canonical forgetful tensor functor. By Tannakian formalism, $H\cong \End(F)$ and  $K=\End(F\circ\iota)$
is a Hopf algebra such that $\A$ is canonically equivalent to $\Rep(K)$. The natural map
$p: H=\End(F)\to \End(F\circ\iota)=K$ is a surjective homomorphism of Hopf algebras.  It is clear
that any embedding  $\iota':\A'\to  \Rep(H)$  with $\iota(\A)=\iota(\A')$ results
in a homomorphism $p':H \to K'$ equivalent to $p$.
\end{proof}

\begin{definition}
\label{trivializable def}
Let $\A$ be a tensor category and let $\mathcal{B}\subset \A$ be a
tensor subcategory. A tensor autoequivalence $\alpha$ of $\A$   is
called {\em trivializable} on $\mathcal{B}$ if the restriction
$\alpha |_\mathcal{B}$ is isomorphic to $\id_\mathcal{B}$ as a
tensor functor.
\end{definition}

We will denote by $\Aut(\A)$ (respectively, $\Aut(\A;\,
\mathcal{B})$) the group of isomorphism classes of tensor
autoequivalences of $\A$ (respectively, tensor autoequivalences of
$\A$ trivializable on $\mathcal{B}$).

Below we consider examples of tensor autoequivalences of $\A =\Rep(H)$,
where $H$ is a finite dimensional Hopf algebra.

\begin{example}
\label{twisted modules1}
A Hopf algebra automorphism  $a:H \xrightarrow{\sim} H$ gives rise to a tensor
autoequivalence $F_a \in \Aut(\Rep(H))$  such that for any $H$-module $V$ one has $F_a(V)=~V$
as a vector space with the action $h\ot v \mapsto a(h)\cdot v,\,h\in H,\, v\in V$.
\end{example}

\begin{definition}
\label{invariant cocycles}
An {\em invariant twist}   \cite{Da}  on $H$  is
an invertible element $J \in H\ot H$ such that  $J\Delta(h) =\Delta(h)J$ for all $h\in H$
and
\[
(J \ot 1)\,(\Delta\ot \id)(J) = (1\ot J)\, (\id\ot \Delta)(J).
\]

An {\em invariant $2$-cocycle} on $H$  \cite{BC} is a convolution
invertible linear map $\mu:H\ot H\to k$ such that  $\mu(x\1 \ot
y\1) x\2 y\2 = \mu(x\2 \ot y\2) x\1 y\1$ and
\[
\mu(x\1 \ot y\1) \mu (x\2 y\2 \ot z) = \mu (y\1 \ot z\1) \mu(x \ot y\2 z\2),
\]
for all $x,\,y,\,z\in H$.
\end{definition}

\begin{remark}
\begin{enumerate}
\item[(i)] If $\sigma : H\ot H\to k$ is an invariant $2$-cocycle
on $H$ then $\sigma^*$ (viewed as an element of $H^*\ot H^*$) is
an invariant twist on $H^*$. Similarly, an invariant twist on $H$
gives rise to an invariant $2$-cocycle on $H^*$
\item[(ii)] Invariant twists on $H$ (respectively, invariant $2$-cocycles on
$H$) form a group under multiplication (respectively, convolution).
\end{enumerate}
\end{remark}

Let $u: H \to k$ be a convolution invertible linear map such that $u(x\1)x\2 = x\1 u(x\2)$ for
all $x\in H$ (i.e., $u$ is a central element of $H^*$). Then
\begin{equation}
\label{Ju}
\mu_u(x\ot y) =  u^{-1}(x\1)\,u^{-1}(y\1)\, u(x\2 \ot y\2),\qquad x,\,y\in H,
\end{equation}
is an invariant $2$-cocycle on $H$. The quotient of the group of
invariant $2$-cocycles on $H$ by the subgroup of  $2$-cocycles of
the form \eqref{Ju} is called the {\em second invariant cohomology
group} of $H$ and is denoted by $\Hinv(H)$.

\begin{example}
\label{twisted modules2}
Any  invariant twist $J \in H\ot H$  gives rise
to  a tensor autoequivalence $F_J \in \Aut(\Rep(H))$ such  that $F_J=\id_{\Rep(H)}$
as an additive functor with  the tensor structure  given by
\begin{equation}
\label{J VU}
J_{V,U} : V\ot U \to V \ot U :  v \ot u \mapsto J(v\ot u),
\end{equation}
where $V,\,U$ are $H$-modules, $v\in V,\,u\in U$.

In a dual way, an invariant  $2$-cocycle $\sigma : H\ot H \to k$
gives rise to a tensor autoequivalence $F_{\sigma}$ of $\Rep(H^*)$
that is isomorphic to $\id_{\Rep(H^*)}$ as an additive functor. If
we view $\Rep(H^*)$ as the category of right $H$-comodules then
the tensor  structure of $F_{\sigma}$  is given by
\begin{equation}
\label{sigma VU} \sigma_{V, U} : V \ot U \to V \ot U: v \ot u
\mapsto \sigma(v_{(1)}\ot u_{(1)}) v_{(0)} \ot u_{(0)},
\end{equation}
where $V,\,U$ are $H$-comodules, $v\in V,\,u\in U$, and $v\mapsto  v_{(0)}\ot v\1,\,
u\mapsto  u_{(0)}\ot u\1$ denote the comodule maps.

Let $\sigma_1,\,\sigma_2$ be invariant $2$-cocycles of $H$. Then $F_{\sigma_1}$ and $F_{\sigma_2}$
are isomorphic autoequivalences of $\Rep(H)$ if and only if $\sigma_1$ and $\sigma_2$ determine
the same class in $\Hinv(H)$.
\end{example}

Let $\C$ be a braided tensor  category with braiding $c_{X,Y}: X
\ot Y \xrightarrow{\sim} Y \ot X$.  For a tensor subcategory $\D
\subset \C$ its {\em centralizer}  $\D'$ is the full tensor subcategory
of $\C$ consisting of  objects $Y$ such that $c_{YX}\circ c_{XY}
=\id_{X\ot Y}$ for all objects $X$ in $\C$ \cite{Mu}.

A braided tensor category $\C$ is called
{\em symmetric} if $\C=\C'$, i.e., if
\[
c_{YX}\circ c_{XY} =\id_{X\ot Y}
\]
for all objects $X,\,Y$ in $\C$.

\begin{remark}
\label{Deligne}
By the result of Deligne \cite{De} any finite symmetric tensor category
is equivalent to the representation category of a finite supergroup.
It was explained in \cite{AEG} that any such a category can be realized as
the representation category of a modified supergroup Hopf algebra $\wedge V \rtimes kG$,
where $G$ is a finite group with a fixed central  element $u$ such that $u^2=1$ and
$V$ is a finite dimensional representation of $G$ on which $u$ acts by $-1$.
The coalgebra structure of $\wedge V \rtimes kG$ is determined by
\[
\Delta(g)=g\ot g,\,\eps(g)=1,\, g\in G, \qquad \Delta(v)= 1\ot v + v\ot  u,\,\eps(v)=0,\, v\in V,
\]
and the antipode is given by $S(g)=g^{-1},\, S(v)=-v$. This category is semisimple if and only if $V=0$.

In this paper we deal with non-semisimple symmetric tensor categories corresponding to
$G=\mathbb{Z}/2\mathbb{Z}$. These are precisely the finite symmetric categories  without
non-trivial Tannakian subcategories.  The corresponding Hopf algebras are described in Section~\ref{En section}
below.

Any finite symmetric tensor category has a unique, up to isomorphism, super-fiber functor
(i.e., a braided tensor functor to the category $\sVec$ of super vector spaces). This functor is identified with
the forgetful tensor functor
\[
\Rep(\wedge V \rtimes kG)\to \sVec.
\]
\end{remark}


\begin{example}
Let $\C= \Rep(H)$.  It is well known that braidings on $\C$ are in
bijection with  {\em quasi-triangular structures} on
$H$. By definition, such a structure on $H$ is an invertible
element $R \in H\ot H$ such that $\Delta^{\op}(h)=R\Delta(h)R^{-1}$
for all $h\in H$ and
\[
(\Delta\ot \id)(R) = R^{13}R^{23},\qquad (\id \ot \Delta)(R) = R^{13}R^{12}.
\]
Here   we denote  $R^{12}=R\ot 1,\, R^{23} =1\ot R$ and $R^{13} =
\sum_i\, a_i \ot 1 \ot b_i$, where  $R=\sum_i \, a_i\ot b_i$.

A quasi-triangular structure is {\em triangular} if $R_{21}=R^{-1}$.
Triangular structures are in bijection with symmetric braidings on $\Rep(H)$.
\end{example}

For a braided tensor category $\C$ let  $\Aut^{br}(\C)$ denote the
group of isomorphism classes of braided autoequivalences of $\C$.

\subsection{The center of a tensor category}
\label{the center section}

For any tensor  category $\A$ let $\Z(\A)$ denote its {\em center}.  Recall that the objects of $\Z(\A)$ are pairs
$(Z,\,\gamma)$ where $Z$ is an object of $\A$  and $\gamma=\{\gamma_X\}_{X\in \A}$, where
\begin{equation}
\label{gX}
\gamma_X: X\ot Z \xrightarrow{\sim} Z\ot X,
\end{equation}
is a natural isomorphism   satisfying certain compatibility conditions.  We will usually simply write $Z$ for $(Z,\,\gamma)$.
Morphisms  and tensor product in $\Z(\C)$ are defined in an obvious way.

\begin{example}
\label{YD as center} The center of the tensor category
$\C=\Rep(H)$, where $H$ is a Hopf algebra, is equivalent to
$\Rep(D(H))$, where the Hopf algebra $D(H)$ is the {\em Drinfeld
double} of $H$. As a coalgebra,
\[
D(H) = H^{*\text{cop}} \ot H,
\]
where $H^{*\text{cop}}$ denotes the co-opposite dual Hopf algebra
of $H$.  The algebra structure of $D(H)$ is given by
\begin{equation} \label{multiplication in D(H)}
(p \ot h) (p' \ot h') = p \Big( h\1 \rightharpoonup p'
\leftharpoonup S^{-1} (h\3) \Big) \ot h\2 h'
\end{equation}
for all $p$, $p' \in H^{*}$, $h$, $h' \in H$, where $(h
\rightharpoonup p \leftharpoonup g) (x) = p (g x h)$, for all $h$,
$g$, $x \in H$, $p \in H^{*}$.   There is a canonical quasi-triangular
structure on $D(H)$:
\[
\mathcal{R} =\sum_i\, (1\ot e_i) \ot (f_i \ot 1),
\]
where $\{e_i\}$ and $\{f_i\}$ are dual bases of $H$ and $H^*$. This
quasitriangular structure
corresponds to braiding \eqref{gX} of $\Z(\Rep(H))$.

It is well known that
representations of $D(H)$ (i.e., objects of $\Z(\Rep(H))$)
can be described in terms of {\em Yetter-Drinfeld modules} over
$H$. That is, any such representation $Z$ has structures of a left
$H$-module $h\ot x \mapsto h\cdot x$ and a right $H$-comodule
$\delta: Z \to Z \ot H,\, \delta(x)= x\0 \ot x\1,\,h\in H,\, x\in
Z$ satisfying the following compatibility condition:
\[ \delta(h \cdot x)= h\2 \cdot
x\0 \ot h\3 x\1 S^{-1}(h\1),\qquad h\in H,\, x\in Z.
\]
Let $_{H} \mathcal{YD}^{H}$ denote the tensor category of
Yetter-Drinfeld modules over $H$.
The central structure \eqref{gX}
for  the Yetter-Drinfeld module $Z$ is given by
\begin{equation}
\label{central structure}
\gamma_V: V\ot Z \to Z \ot V,\qquad
\gamma_V (v \ot x) = x\0 \ot x\1 \cdot v,\qquad v \in V,\, x\in Z.
\end{equation}
\end{example}

\subsection{The Brauer-Picard group of a tensor category}
\label{BrPic section}

Let $\A$ be a finite tensor category. The notion of a
tensor product $ \bt_\A$ of $\A$-bimodule categories was
introduced in \cite{ENO}.  With respect to this product
the equivalence classes of $\A$-bimodule categories form a monoid. The
unit of this monoid is the regular $\A$-bimodule category $\A$.
An $\A$-bimodule category $\M$ is called {\em invertible} if there
is an $\A$-bimodule category $\N$ such that $\M \bt_\A \N \cong
\A$ and $\N \bt_\A \M \cong \A$.  By definition, the {\em
Brauer-Picard} group of $\A$ is the group $\BrPic(\A)$ of
equivalence classes of invertible $\A$-bimodule categories.

The Brauer-Picard group is an important invariant of a tensor
category. It is used, in particular, in the classification of
extensions of tensor  categories \cite{ENO}.

\subsection{$\Ext$-spaces and actions of autoequivalences on them}
\label{ext prelim section}

Let $\A$ be a  $k$-linear abelian category  and let $V,\, U$ be objects of $\A$.
An {\em extension} of $U$ by $V$ is a short exact sequence
\begin{equation}
\label{an extension}
0\to V \xrightarrow{i} E \xrightarrow{p} U \to 0.
\end{equation}
We will consider extensions up to the following equivalence relation. Two extensions
$0\to V \xrightarrow{i} E \xrightarrow{p} U \to 0$ and $0\to V \xrightarrow{i'} E' \xrightarrow{p'} U \to 0$
are {\em equivalent} if there is an isomorphism $\phi:E\xrightarrow{\sim}E'$ such that the following diagram
commutes:
\begin{equation}
\xymatrix{
0 \ar[r] & V \ar[r]^{i} \ar@{=}[d]  & E \ar[d]^{\phi}   \ar[r]^{p} & U \ar@{=}[d] \ar[r] & 0 \\
0 \ar[r] & V \ar[r]^{i'}  & E'    \ar[r]^{p'}  & U  \ar[r] & 0.
}
\end{equation}
Equivalences classes of extensions form a $k$-vector space
$\Ext_\A^1(U,\,V)$. The addition is given by the usual operation
of the Baer sum.  The zero element of $\Ext_\A^1(U,\,V)$ is the
split extension $0\to V \to V \oplus U \to
U \to 0$. For $\lambda \in k^\times$ the $\lambda$-multiple of the
class of extension \eqref{an extension} is  the class of the
extension
\[
0\to V \xrightarrow{\lambda^{-1}\,i} E \xrightarrow{p} U \to 0.
\]
Furthermore,  the  extension $0\to V
\xrightarrow{\lambda\,i} E \xrightarrow{\mu\,p} U \to 0$, where
$\lambda,\, \mu \in k^\times$, is equivalent to $0\to V
\xrightarrow{\lambda\mu\,i} E \xrightarrow{p} U \to 0$.

\begin{remark}
\label{skew-prims}
Let $H$ be a Hopf algebra and $\gamma$ and $\eta$ be two
1-dimensional representations of $H$. If $\textnormal{P}_{\gamma,
\eta} (H^{*})$ denotes the set of $(\gamma, \eta)$-primitive
elements of $H^{*}$, i.e. elements $\xi \in H^{*}$ such that
$\Delta_{H^{*}} (\xi) = \gamma \ot \xi + \xi \ot \eta$, then
$$
\Ext_{\Rep(H)}^1(\eta, \gamma) \cong \textnormal{P}_{\gamma, \eta}
(H^{*}) / k (\gamma - \eta).
$$
Indeed, let $0 \to \gamma \xrightarrow{i} E \xrightarrow{p} \eta
\to 0$ be an extension of $\eta$ by $\gamma$ and let $e_{1}$,
$e_{2} \in E$ be such that $e_{1} = i(1)$ and $p(e_{2}) = 1$. Then
$\{e_{1}, e_{2}\}$ is a $k$-basis of $E$ such that $h \cdot e_{1}
= \gamma (h) e_{1}$ and $h \cdot e_{2} = \xi(h) e_{1} + \eta (h)
e_{2}$, for all $h \in H$ and some $\xi \in H^{*}$. In fact, $\xi
\in \textnormal{P}_{\gamma, \eta} (H^{*})$, since
$$
(hl) \cdot e_{2} = h \cdot (\xi(l) e_{1} + \eta(l) e_{2}) =
(\gamma(h) \xi(l) + \xi(h) \eta (l))e_{1} + \eta (hl) e_{2}
$$
implies that $\xi (hl) = \gamma(h) \xi(l) + \xi(h) \eta (l)$, for
all $h$, $l \in H$, whence $\Delta(\xi) = \gamma \ot \xi + \xi \ot
\eta$. If $e'_{2}$ and $\xi'$ are such that $p (e'_{2}) = 1$ and
$h \cdot e'_{2} = \xi'(h)e_{1} + \eta(h) e'_{2}$ then $\xi' - \xi
\in k (\gamma - \eta)$. Indeed, there exists $a \in k$ such that
$e_{2} - e'_{2} = a e_{1}$ and action of $h \in H$ on this
relation yields $\xi - \xi' = a (\gamma - \eta)$. Similarly, the
equivalence class of $\xi$ modulo $k (\gamma - \eta)$ remains the
same if we pass to an extension equivalent to $E$. Thus, the map
$\Ext_{\Rep(H)}^1(\eta, \gamma) \ni [E] \mapsto \widehat{\xi} \in
\textnormal{P}_{\gamma, \eta} (H^{*}) / k (\gamma - \eta)$,
where $\widehat{\xi}$  denotes the class of $\xi$, is well
defined and is easily seen to be a $k$-vector space isomorphism.
\end{remark}

\begin{proposition}
\label{action on Ext}
Let $\A$ be a tensor category and
let $U,\,V$ be simple objects of $\A$ such that $\alpha(V)=V$ and $\alpha(U)=U$ for all  tensor autoequivalences
$\alpha:\A \to \A$.  Then isomorphisms
\begin{equation}
\label{alpha on ext} \rho(\alpha): \Ext_\A^1(U,\,V)
\xrightarrow{\sim} \Ext_\A^1(\alpha(U),\,\alpha(V)) =
\Ext_\A^1(U,\,V),
\end{equation}
where the image of extension \eqref{an extension} under $\rho(\alpha)$ is
\[
0\to V \xrightarrow{\alpha(i)} \alpha(E) \xrightarrow{\alpha(p)} U \to 0,
\]
gives rise to a projective representation of $\Aut(\A)$ on $\Ext_\A^1(U,\,V)$.
\end{proposition}
\begin{proof}
Let $\alpha,\, \alpha':\A \to A$ be tensor autoequivalences  and let $\phi:\alpha\to \alpha'$ be a tensor
isomorphism between them (so that $\alpha,\,\alpha'$ determine the same element of $\Aut(\A)$).
We have an isomorphism of extensions:
\begin{equation*}
\xymatrix{
0 \ar[r] & V \ar[r]^{\alpha(i)} \ar^{\phi_V}[d]  & \alpha(E) \ar[d]^{\phi_E}   \ar[r]^{\alpha(p)} & U \ar^{\phi_U}[d] \ar[r] & 0 \\
0 \ar[r] & V \ar[r]^{\alpha'(i)}  & \alpha'(E) \ar[r]^{\alpha'(p)}
& U  \ar[r] & 0, }
\end{equation*}
where $\phi_V,\,\phi_U$ are non-zero scalars and $\phi_E$ is an isomorphism. Thus, the equivalence classes of
$\rho(\alpha)$ and $\rho(\alpha')$ differ by the scalar $\phi_V\phi_U^{-1}$. Hence, the map
\[
\rho:\Aut(\A)\to PGL(\Ext_\A^1(U,\,V))
\]
is well defined. It is clear that this map is a group homomorphism.
\end{proof}

\begin{remark}
Presence of equalities rather than isomorphisms in  the hypothesis
of Proposition~\ref{action on Ext} may look unnatural. This can be resolved by replacing $\A$ with
an equivalent skeletal category (i.e., a category in which isomorphic objects are equal).
In any event, we will only  apply  Proposition~\ref{action on Ext}  in a simple situation when $U,\, V$
are one-dimensional representations (i.e., linear characters) of a Hopf algebra.
\end{remark}

\section{Braided autoequivalences of the center}

\subsection{Isomorphism $\BrPic(\A) \simeq \Aut^{br}(\Z(\A))$}


Let $\A$ be a tensor  category and let $\M$ be an invertible $\A$-bimodule category.
One assigns to $\M$ a braided autoequivalence $\Phi_\M$ of $\Z(\A)$ as follows.
Note that  $\Z(\A)$ can be identified with the category of $\A$-bimodule endofunctors of $\M$ in two ways:
via the functors $Z \mapsto Z\ot -$ and $Z \mapsto - \ot Z$.  Define $\Phi_\M$ in such a way that
there is an isomorphism of $\A$-bimodule functors
\begin{equation}
\label{PhiM}
\Phi_\M(Z) \ot - \cong -\ot Z
\end{equation}
for all $Z\in \Z(\A)$.

The following result was established in \cite{ENO, DN}.

\begin{theorem}
\label{BrPic=Aut-br}
Let $\A$ be a tensor  category. The assignment

\begin{equation}
\label{main iso}
\Phi: \BrPic(\A) \xrightarrow{\sim} \Aut^{br}(\Z(\A)): \M \mapsto \Phi_\M
\end{equation}
is an isomorphism.
\end{theorem}

\begin{example}
Let $A$ be a finite Abelian group. Then Theorem~\ref{BrPic=Aut-br} implies that
\begin{equation}
\BrPic(\Vec_A)  \cong O(A \oplus \widehat{A},\, q),
\end{equation}
where $O(A \oplus \widehat{A},\,q)$
is the group of automorphisms of $A  \oplus \widehat{A}$ preserving the
canonical quadratic form
\[
q(a,\,\chi) =\chi(a),\qquad a\in A,\, \chi\in \widehat{A}.
\]
\end{example}

Let $\C$ be a braided tensor category with braiding $c_{X,Y}: X\ot Y \xrightarrow{\sim} Y \ot X$.
Then  $\C$ is embedded into $\Z(\C)$ via
\[
X\mapsto (X,\, c_{-,X}).
\]
In what follows we will identify $\C$ with a tensor  subcategory of $\Z(\C)$ (the image of this embedding).
The braiding of $\C$ allows to view
left $\C$-module categories  as $\C$-bimodule categories (analogously to how modules over a commutative
ring can be viewed as bimodules).  Invertible left $\C$-module categories form a subgroup
$\Pic(\C) \subset \BrPic(\C)$ called the {\em Picard} group of $\C$.

\begin{remark}
It follows from \cite{ENO} that
\[
\BrPic(\C) \cong \Pic(\Z(\C)).
\]
\end{remark}

Let $\Aut^{br}(\Z(\C);\, \C) \subset \Aut^{br}(\Z(\C))$ be the
subgroup consisting of braided autoequivalences of $\Z(\C)$ that
restrict to a trivial autoequivalence of $\C$. The following
result was established in \cite{DN}.

\begin{theorem}
\label{image of Pic}
The image of $\Pic(\C)$ under isomorphism \eqref{main iso} is  $\Aut^{br}(\Z(\C);\, \C)$.
\end{theorem}


%
%
%

\subsection{Homomorphism  $\Aut(\A) \to \Aut^{br}(\Z(\A))$ }

For any tensor category $\A$ there is an   induction homomorphism
\begin{equation}
\label{induction from Aut}
\Gamma: \Aut(\A) \to \Aut^{br}(\Z(\A)) : \alpha \mapsto \Gamma_\alpha,
\end{equation}
where $\Gamma_\alpha(Z,\, \gamma) = (\alpha(Z),\, \gamma^\alpha)$ and   $\gamma^\alpha$
is defined by the following commutative diagram
\begin{equation}
 \xymatrix{
 X\otimes \alpha(Z)\ar[rr]^{\gamma^\alpha_X}\ar[d] & &  \alpha(Z)\otimes X\ar[d] \\
 \alpha(\alpha^{-1}(X)) \otimes \alpha(Z)\ar[d]_{J_{\alpha^{-1}(X),Z}} & & \alpha(Z) \otimes \alpha(\alpha^{-1}(X))\ar[d]^{J_{Z,\alpha^{-1}(X)}}\\
 \alpha(\alpha^{-1}(X)\otimes Z)\ar[rr]^{\alpha(\gamma_{\alpha^{-1}(X)})} & & \alpha(Z\otimes \alpha^{-1}(X)).
 }
\end{equation}
Here $\alpha^{-1}$ is a quasi-inverse of $\alpha$ and $J_{X,Y}: \alpha(X)\ot \alpha(Z) \xrightarrow{\sim} \alpha(X\ot Z)$
is the tensor functor structure of $\alpha$.

For any invertible object $Z\in \A$ let $\ad(Z)$ denote the tensor autoequivalence of $\A$ given by
\[
\ad(Z)(X)= Z \ot X \ot Z^*.
\]
Thus, we  have a homomorphism
\begin{equation}
\label{ad}
\ad: \Inv(\A) \to \Aut(\A)
\end{equation}


Let
\[
F_\A: \Z(\A)\to \A : (Z,\, \gamma) \mapsto Z
\]
denote the canonical forgetful functor. Let $I_\A: \A \to \Z(\A)$
denote a functor right adjoint to $F_\A$. It is known that
$I_\A(\be)$ is a commutative algebra in $\Z(\A)$ and that $\A$ is
isomorphic to the category of $I_\A(\be)$-modules as a tensor
category \cite{DMNO}.

\begin{proposition}
\label{Image of Gamma}
The image of homomorphism \eqref{induction from Aut} consists of (isomorphism
classes of) all braided
autoequivalences $\beta:\Z(\A) \to\Z(\A)$ such that $\beta(I_\A(\be)) \cong I_\A(\be)$ as algebras in $\Z(\A)$.
\end{proposition}
\begin{proof}
It is clear from definitions that  for any  tensor autoequivalence $\alpha: \A \to \A$
one has $\alpha\circ F_\A \cong  F_\A \circ \Gamma_\alpha$. Taking the adjoints of both
sides and applying them to $\be$ we obtain algebra isomorphism $\Gamma_\alpha(I_\A(\be)) \cong~I_\A(\be)$.

Conversely, assume that  $\beta:\Z(\A) \to\Z(\A)$ is such that there is an algebra isomorphism $\varphi:
I_\A(\be) \xrightarrow{\sim}
\beta(I_\A(\be))$.  Then $\beta$ induces a tensor autoequivalence $\alpha$ of the category of $I_\A(\be)$-modules
in $\Z(\A)$ defined by $Z\mapsto  \beta(Z)$ with the action of $I_\A(\be)$ given by
\[
I_\A(\be) \ot \beta(Z) \xrightarrow{\varphi\ot \id_{\beta(Z)}}   \beta(I_\A(\be)) \ot \beta(Z)   \cong \beta({I_\A(\be) \ot Z})
\xrightarrow{\mu_Z} \beta({Z}),
\]
where $\mu_Z$ denotes the action of $I_\A(\be)$ on $Z$.
Since the category of  $I_\A(\be)$-modules  in $\Z(\A)$ is identified with $\A$ we get $\alpha\in \Aut(\A)$.
Since $\Z(\A)$  is identified with the center of the category of $I_\A(\be)$-modules via the free module functor
$Z\mapsto I_\A(\be) \ot Z$ \cite{DMNO} we see that $\Gamma_\alpha \cong \beta$.
\end{proof}


Let $\A =\Rep(H)$, where $H$ is a finite dimensional Hopf algebra.
As in Example~\ref{YD as center}, we identify $\Z(\Rep(H))$ with the category of Yetter-Drinfeld modules over $H$.
Below we describe braided
autoequivalences of $\Z(\Rep(H))$ induced from tensor autoequivalences of $\Rep(H)$, see Examples~\ref{twisted modules1}
and \ref{twisted modules2}.

\begin{example}
\label{twisted YD-modules0}
Let $a:H \xrightarrow{\sim} H$ be a Hopf algebra automorphism and let $F_a \in \Aut(\Rep(H))$ be the
corresponding tensor autoequivalence.
Then $\Gamma_{F_a}(Z) = Z$ as a vector space with the action and coaction given by
\begin{equation}
\label{Delta a} h\ot x \mapsto a(h)\cdot x,\qquad \rho_a (x) = x\0
\ot a(x\1)\qquad  h\in H,\, x\in Z.
\end{equation}
\end{example}

\begin{example}
\label{twisted YD-modules1}
Let $J \in H\ot H$ be an invariant twist and  let $F_J \in \Aut(\Rep(H))$ denote the
corresponding tensor autoequivalence (see Example~\ref{twisted modules2}).
Then $\Gamma_{F_J}(Z)=Z$ as an $H$-module
with the coaction given by
\begin{equation}
\label{Delta J}
\rho^{J} (x) = (J^{-1})^{2} \cdot (J^{1} \cdot
x)\0 \ot (J^{-1})^{1} (J^{1} \cdot x)\1 J^{2}, \qquad x\in Z.
\end{equation}
Here $J^1\ot J^2$ stands for $J$ and $J^{-1} \ot J^{-2}$ for the inverse of $J$.
Note that formula \eqref{Delta J} appeared in
\cite{CZ}.
\end{example}

\begin{example}
\label{twisted YD-modules2} Let $\sigma \in (H \ot H)^{*}$ be an
invariant $2$-cocycle on $H$. The dual map $\sigma^{*}$ can be
seen as an invariant twist on $H^*$ and, as such, it gives rise to
an autoequivalence $F_{\sigma}$ of $\Rep (H^{*})$. By Example
\ref{twisted YD-modules1}, this induces an autoequivalence
$\Gamma_{F_{\sigma}}$ of $_{{H}^{*}} \mathcal{YD} ^{H^{*}}$. Since
$_{H} \mathcal{YD}^{H}$ and $_{{H}^{*}} \mathcal{YD} ^{H^{*}}$ are
tensor equivalent  via the functor that dualizes module and
comodule structures, we obtain an autoequivalence of $_{H}
\mathcal{YD}^{H}$, which we shall still denote by
$\Gamma_{F_{\sigma}}$. If $V \in \, _{H} \mathcal{YD}^{H}$ then
$\Gamma_{F_{\sigma}} (V) = V$ as an $H$-comodule, with the
$H$-action given by
\begin{equation}
\label{effect of invariant cycles} h \cdot v = \sigma^{-1} \Big(
\big( h_{(2)} \cdot v_{(0)} \big)_{(1)} \ot  h_{(1)} \Big) \sigma
(h_{(3)} \ot  v_{(1)}) \big( h_{(2)} \cdot v_{(0)} \big)_{(0)},
\quad h \in H, v \in V.
\end{equation}
The invertible $\Rep(H)$-bimodule category $\M_\sigma$ corresponding to $\Gamma_{F_{\sigma}}$
under isomorphism \eqref{main iso} can be described as follows.
Let $H_\sigma$ be the algebra $H$ with the multiplication twisted by $\sigma$, i.e.,
$H_\sigma=H$ as a vector space and
\[
a\cdot b := \sigma(a\1 \ot b\1) a\2 b\2 = a\1 b\1 \sigma(a\2 \ot
b_2),\qquad a,\,b\in H_\sigma.
\]
Then $\M_\sigma$ is the category of $H_\sigma$-modules with the obvious $\Rep(H)$-module actions.
The fact that it is a $\Rep(H)$-bimodule category follows from invariance of $\sigma$.
\end{example}

\begin{remark}
\label{iotas}
Autoequivalences described in  Examples~\ref{twisted YD-modules0}, \ref{twisted YD-modules1}, \ref{twisted YD-modules2}
give rise to  group homomorphisms
\begin{eqnarray*}
\iota_1 &:&  \Aut_{\text{Hopf}}(H)\to \Aut^{br}\big (\Z(\Rep(H)) \big),\\
\iota_2 &:&  \Hinv(H) \to \Aut^{br} \big( \Z(\Rep(H)) \big),\\
\iota_3 &:&  \Hinv(H^*) \to \Aut^{br} \big( \Z(\Rep(H)) \big).
\end{eqnarray*}
\end{remark}

Let $(H,\, R)$ be a quasi-triangular Hopf algebra. Then  $\Rep(H)$ is embedded into  $_{H} \mathcal{YD}^{H}$
by defining the coaction on an $H$-module $V$ via
\begin{equation}
\label{v--> Rv}
\delta(v)= R^{21}(v\ot 1),\qquad v\in V.
\end{equation}

\begin{proposition}
\label{Hinv to Pic}
Let $(H,\, R)$ be a quasi-triangular Hopf algebra. Let $\sigma$ be an invariant $2$-cocycle on $H$
such that
\begin{equation}
\label{R-invariance of sigma}
\sigma(R^1\ot h) R^2=\eps(h),\quad \sigma^{-1}(h \ot R^1)R^2=\eps(h)
\end{equation}
for all $h\in H$, where $R=R^1\ot R^2$.  Then $\Gamma_{F_\sigma}$ belongs to the Picard group
of $\Rep(H)$ (corresponding to the braiding defined by $R$) via identification  from Theorem~\ref{image of Pic}.
\end{proposition}
\begin{proof}
It suffices to show that $\Gamma_{F_\sigma}$ belongs to
$\Aut^{br}( {}_{H} \mathcal{YD}^{H},\, \Rep(H))$ (recall that $\Z(\Rep(H))\cong {}_{H} \mathcal{YD}^{H}$
as a braided tensor category).
For any $H$-module $V$  (viewed as a Yetter-Drinfeld module via \eqref{v--> Rv})
we have $\Gamma_{F_\sigma}(V)=V$ as an $H$-comodule, while its $H$-module structure
is computed using  \eqref{effect of invariant cycles} :
\[
h\cdot v =\sigma(R^1 \ot h\1) \, \sigma^{-1}(h\3\ot r^1)\, R^2h\2r^2\cdot v,\qquad h\in H,\, v\in V.
\]
Here $r$ stands for another copy of $R$.  It is clear from the last equation
that the invariance condition \eqref{R-invariance of sigma} implies   $\Gamma_{F_\sigma}|_{\Rep(H)} = \id_{\Rep(H)}$.
\end{proof}

\subsection{Induction from subcategory}
\label{induction section}

Let $\C$ be a braided tensor category and let $\D \subset \C$ be a  tensor subcategory. Given an invertible
$\D$-module category $\M$ we can consider the induced $\C$-module category  $\C \bt_\D \M$.
It is clear that the latter category is invertible and that the assignment
\begin{equation}
\label{IndDC}
\Ind_\D^\C: \Pic(\D) \to \Pic (\C) : \M \mapsto \C \bt_\D \M
\end{equation}
is a group homomorphism\footnote{More precisely, this is a monoidal functor
between monoidal groupoids.}.

Let $A$ be an algebra in $\D$ such that $\M$ is equivalent to the
category of $A$-modules in $\D$. Then $\Ind_\D^\C(\M)$ is the
category of $A$-modules in $\C$.

Note that homomorphism \eqref{IndDC} is not injective in general.
Let
\[
\C =\bigoplus_{\alpha\in \Sigma}\,\C_\alpha
\]
be the decomposition of $\C$ into a direct sum of $\D$-module
subcategories (this decomposition exists and is unique by
\cite{EO}).

Let $\Sigma_0 =\{ \alpha \in \Sigma  \mid \C_\alpha \mbox { is an
invertible $\D$-module category} \}$.

\begin{proposition}
The kernel of homomorphism \eqref{IndDC} is precisely the set  of the equivalence classes of categories
$\C_\alpha,\, \alpha\in \Sigma_0$.
\end{proposition}
\begin{proof}
Suppose that $\M$ is an invertible $\D$-module category such that
\begin{equation}
\label{CM=C}
\C \bt_\D \M \cong \C
\end{equation}
as a $\C$-module category.  From the  $\D$-module decomposition of both sides  of \eqref{CM=C} we obtain
\[
\bigoplus_{\alpha\in \Sigma}\,(\C_\alpha \bt_\D \M) \cong \bigoplus_{\beta\in \Sigma}\,\C_\beta.
\]
Taking $\alpha$ such that $\C_\alpha =\D$ we conclude that $\M \cong \C_\beta$ for some $\beta\in \Sigma_0$.

Conversely, suppose that  $\M \cong \C_\beta$ for some $\beta\in \Sigma_0$.  Note that
\[
\E= \bigoplus_{\alpha\in \Sigma_0}\,\C_\alpha
\]
is a  tensor subcategory of $\C$. It is group graded with the trivial component $\D$. Thus,
$\Ind_\D^\E \M \cong \E$ and, hence, $\Ind_\D^\C\M = \Ind_\E^\C(\Ind_\D^\E \M )\cong \C$.
\end{proof}

\begin{proposition}
\label{image of induced}
Let $\C$ be a braided tensor category and let $\D\subset\C$ be a tensor subcategory.
The image  of the composition
\[
\Pic(\D) \xrightarrow{\Ind_{\D}^{\C}} \Pic(\C) \to \Aut^{br}(\C)
\]
is contained in $\Aut^{br}(\C;\, \D')$. Here $\D'$ denotes the centralizer of $\D$ in $\C$.
\end{proposition}
\begin{proof}
Let $\M \in \Pic(\D)$ be an invertible $\D$-module category and
let $A\in \D$ be an algebra such that $\M$ is identified with the
category of $A$-modules in $\D$.  When we view $A$ as an algebra
in $\C$ the corresponding element $\partial_\M$ of $\Aut^{br}(\C)$
is determined by the existence of a natural tensor isomorphism
\[
A \ot X \cong \partial_\M(X)\ot A,\qquad X\in \C,
\]
of $A$-modules \cite{DN}.  When $X$ centralizes $\D$ (and, hence, centralizes $A$) we get a natural
tensor isomorphism $\partial_\M(X)\cong X$, i.e., $\partial_\M|_{\D'} \cong \id_{\D'}$.
\end{proof}


\subsection{Action on the categorical Lagrangian Grassmannian}

Let $\A$ be a finite tensor category.

\begin{proposition}
\label{braiding and embeddings}
There is a bijection between the set of braidings on $\A$ and the set of tensor embeddings $\iota: \A \hookrightarrow  \Z(\A)$
such that $F\circ \iota(\A) = \A$. Namely, the braiding $c_{X,Z}: X\ot Z \xrightarrow{\sim} Z \ot X$ corresponds to the embedding
$\iota_c: Z \mapsto (Z,\, c_{-,Z})$.
\end{proposition}
\begin{proof}
This is clear  since every embedding $\iota: \A \hookrightarrow  \Z(\A)$
such that $F\circ \iota(\A) = \A$ yields a braiding on $\A$.
\end{proof}


\begin{definition}
\label{def lagr subc}
Let $\A$ be a tensor category. A tensor subcategory $\L \subset \Z(\A)$ is called {\em Lagrangian} if $\L'=\L$.
\end{definition}

For a braided tensor category $\C$ let $\mathbb{L}_0(\C)$ denote
the set of all Lagrangian subcategories $\L \hookrightarrow
\Z(\C)$ such that $\L \cong \C$ as a braided tensor category. This
is a subset of the categorical Lagrangian Grassmannian introduced
in \cite{NR}. Note that the group $\Aut^{br}(\Z(\C))$ acts on
$\mathbb{L}_0(\C)$ by permutation of subcategories:
\[
\alpha\cdot \L = \alpha(\L),\qquad \alpha\in \Aut^{br}(\Z(\C)),\, \L \in \mathbb{L}_0(\C).
\]

\begin{proposition}
\label{stabilizer of C} Let $\C$ be a braided tensor category with
braiding $c_{X,Z}: X\ot Z \xrightarrow{\sim} Z \ot X$ and let
$\iota_c: \C \hookrightarrow \Z(\C)$, $Z \mapsto (Z,\, c_{-,Z})$
be the corresponding  central embedding. Then the stabilizer of
$\iota_c(\C)$ in $\Aut^{br}(\Z(\C))$ is
\[
\St(\iota_c(\C))\cong  \Pic(\C)\rtimes \Aut^{br}(\C).
\]
\end{proposition}
\begin{proof}
The proof is the same as  \cite[Proposition 6.8]{NR}.
\end{proof}

\section{The Hopf algebra $E(n)$}

\subsection{Definition and basic properties}
\label{En section}

Here we recall a description of  a class of non-semisimple Hopf algebras with symmetric
representation categories, cf.\ Remark~\ref{Deligne}.

\begin{definition}
\label{E(n)}
Let $n$ be a positive integer and $E(n)$ denote the Hopf algebra
with generators $c$, $x_{1}, \dots, x_{n}$, relations
$$
c^{2} = 1, \quad x_{i}^{2} = 0, \quad cx_{i} = -x_{i} c, \quad
x_{i}x_{j} = - x_{j}x_{i}, \quad i, j = 1, \dots n,
$$
and comultiplication, counit and antipode given by
\begin{alignat*}{6}
\Delta (c) & = c \ot c,                        & \quad \varepsilon (c)     &= 1,      & \quad S(c)     & = c  \\
\Delta (x_{i}) & = 1 \ot x_{i} + x_{i} \ot c,  & \quad \varepsilon
(x_{i}) &= 0,      & \quad S(x_{i}) & = cx_{i}
\end{alignat*}
for all $i = 1, \dots, n$.
\end{definition}
The Hopf algebras $E(n)$ were first introduced by Nichols in \cite{N} and
studied in \cite{BDG}, \cite{CD} \cite{PvO1}, \cite{PvO2} and
\cite{CC}. They are pointed Hopf algebras with coradical $kC_{2}$.
Here $C_{2} = \{1, c\}$, and $k$-basis $\{c^{i} x_{P} \mid i = 0,
1, \,\, P \subseteq \{1, \dots, n\} \}$, where, for a subset $P = \{i_{1},
i_{2}, \dots, i_{s}\} \subseteq \{1, 2, \dots, n\}$ such that
$i_{1} < i_{2} < \cdots <i_{s}$, we denote
$x_{P}~=~x_{i_{1}}x_{i_{2}} \cdots x_{i_{s}}$ and $x_{\emptyset} = 1$. We
have $\dim_k(E(n)) = 2^{n+1}$.

For  $P$ as above and  a subset $F =
\{i_{j_{1}}, \dots, i_{j_{r}}\}$ of $P$ define
$$
S (F, P) = \left\{ \begin{array}{lll}
(j_{1} + \cdots + j_{r}) - r (r + 1)/2 & \textnormal{if} & F \neq \emptyset \\
0                                              & \textnormal{if} & F = \emptyset.
\end{array} \right.
$$
Letting $|F|$ denote the number of elements of $F$, we have the
following formula:
$$
\Delta (x_{P}) = \sum_{F \subseteq P} (-1)^{S (F, P)} x_{F} \ot
c^{|F|} x_{P \setminus F}.
$$
The group of Hopf automorphisms of $E(n)$ was computed in
\cite{PvO1}. We have
\begin{equation}
\label{Aut En}
\Aut_{\text{Hopf}} (E(n)) \simeq \GL_{n}(k),
\end{equation}
with the automorphism corresponding to $T = (t_{ij}) \in
\GL_{n}(k)$, being given by $c \mapsto c$ and $x_{i} \mapsto
\sum_{j} t_{ji} x_{j}$, for all $i = 1, \dots, n$. The inner automorphisms of $E(n)$
correspond to $T =\pm I_n$.


Quasi-triangular structures on
$E(n)$ were classified in \cite{PvO1}. They are parameterized by
the set $M_{n}(k)$ of $n$-by-$n$ matrices with coefficients in $k$ and they
are given as follows. First, if $A = (a_{ij}) \in M_{n}(k)$ and $P
= \{i_{1}, i_{2}, \dots, i_{r}\}$ and $F = \{j_{1}, j_{2}, \dots,
j_{r}\}$ are subsets of $\{1, 2, \dots, n\}$ such that $i_{1} <
i_{2} < \cdots < i_{r}$ and $j_{1} < j_{2} < \cdots < j_{r}$ then
we denote by $[A]_{P, F}$ the $r \times r$ minor obtained from $A$
by intersecting the rows $i_{1}, \dots, i_{r}$ with the columns
$j_{1}, \dots, j_{r}$. The $R$-matrix corresponding to $A$ was
described in \cite[Remark 2]{PvO1} as:
\begin{align*}
R_{A} = & \frac{1}{2} (1 \ot 1 + 1 \ot c + c \ot 1 - c \ot c) +
\frac{1}{2} \sum_{|P| = |F|} (-1)^{\frac{|P| (|P| - 1)}{2}} [A]_{P, F} \times \\
& \times ( x_{P} \ot c^{|P|}x_{F} + c x_{P} \ot c^{|P|}x_{F} +
x_{P} \ot c^{|P| + 1}x_{F} - cx_{P} \ot c^{|P| + 1} x_{F}),
\end{align*}
where the sum is over all non-empty subsets $P$ and $F$ of $\{1,
\dots, n\}$. We will use the following equivalent expression for
$R_{A}$:
\begin{align*}
R_{A} = & \frac{1}{2} \sum_{i = 0}^{n} (-1)^{\frac{i (i - 1)}{2}}
\sum_{|P| = |F| = i} [A]_{P, F} ( x_{P} \ot x_{F} + x_{P} \ot cx_{F} + \\
&  + (-1)^{i} c x_{P} \ot x_{F} + (-1)^{i + 1} cx_{P} \ot c x_{F})
\end{align*}
where the sum is over all subsets $P$ and $F$ of $\{1, \dots, n\}$
and the convention is that $[A]_{\emptyset, \emptyset} = 1$. It
was shown in \cite{CC} that the quasi-triangular structure $R_{A}$
is triangular if and only if $A$ is symmetric.

Let $\C_n :=\Rep(E(n))$.

\begin{remark}
The  category $\C_n$ with symmetric braiding is equivalent to the representation
category of a finite supergroup $ \wedge k^n \rtimes \mathbb{Z}/2\mathbb{Z}$. It is
the most general example of a non-semisimple symmetric tensor category without
non-trivial Tannakian subcategories.
\end{remark}

\begin{proposition}
\label{Aut-br Cn}
$\Aut^{br}(\C_n) \cong GL_n(k)/ \{ \pm I_n\}$.
\end{proposition}
\begin{proof}
By \cite{De} the symmetric category $\C_n$ has a unique, up to isomorphism, braided tensor functor to $\sVec$.
Let $F$ denote the composition of this functor with the forgetful functor $\sVec\to \Vec$.  Then $E(n)\cong\End(F)$.
Since every braided tensor auto\-equivalence of $\C_n=\Rep(E(n))$ preserves $F$ it must come from a Hopf automorphism
of $E(n)$.  By  \eqref{Aut En} we have $\Aut_{\text{Hopf}}(E(n))= GL_n(k)$. Tensor auto\-equivalences of $\C_n$ isomorphic to
the identity functor  come from inner automorphisms of $E(n)$.  The statement follows from the observation that the
group of inner Hopf automorphisms of $E(n)$ is generated by the conjugation by $c$ and is isomorphic to $\{ \pm I_n \}$.
\end{proof}

\subsection{Invariant $2$-cocycles and invariant twists on $E(n)$}
\label{inv cocycles on En section}

It was shown in \cite{BC}, that the second invariant cohomology group $\Hinv(E(n))$
is isomorphic to $\Sym_{n}(k)$, which we can identify with the additive group of
lower triangular $n \times n$ matrices with entries in $k$. A
representative of  the cohomology class corresponding to $M =
(m_{ij}) \in \Sym_{n} (k)$ is the invariant 2-cocycle $\sigma_{M} : E(n)
\otimes E(n) \to k$ defined by:
$$
\sigma_{M} (c\ot  c) = 1,  \quad \sigma_{M} (x_{i}\ot  x_{j}) = m_{ij},
\quad i, j = 1, \dots, n,
$$
$$
\sigma_{M} (x_{P}\ot x_{Q}) = \sigma_{M} (cx_{P}\ot x_{Q})  =
(-1)^{|P|} \sigma_{M} (x_{P}\ot  cx_{Q}) = (-1)^{|P|} \sigma_{M}
(cx_{P}\ot cx_{Q})
$$
for all $P$, $Q \subseteq \{1, \dots, n\}$,
\begin{align*}
\sigma_{M} (x_{P}\ot  x_{Q}) = 0 \quad \textnormal{ if } \quad |P|
\neq |Q|,
\end{align*}
and some recurrence formula allowing to compute $\sigma_{M}
(x_{P}, x_{Q})$ when $|P| = |Q|$. In particular, we have
$\sigma_{M} (c^{i}x_{k}\ot c^{j}x_{l}) = (-1)^{j} m_{kl}$, for all
$i, j = 0, 1$ and $k, l = 1, \dots, n$.

Since $E(n)$ is self-dual, an Hopf algebra isomorphism $E(n) \to
E(n)^{*}$ being given by $1 \mapsto 1^{*} + c^{*}$, $c \mapsto
1^{*} - c^{*}$, $x_{i} \mapsto x_{i}^{*} + (cx_{i})^{*}$, where
$\{ (c^{i}{x_{P}})^{*} \}_{i, P}$ is the dual basis of
$\{c^{i}x_{P}\}_{i, P}$, we obtain, by duality, that the invariant dual
cohomology group of $E(n)$, $\Hinv(E(n)^*)$,
is also isomorphic to $\Sym_{n} (k)$. A representative for the
cohomology class corresponding to $M = (m_{ij}) \in \Sym_{n} (k)$ is
the invariant twist
$$
J_{M} = \frac{1}{4} \sum_{i, j, P, Q} \sigma_{M} (c^{i}x_{P}\ot
c^{j} x_{Q}) (x_{P} + (-1)^{i} cx_{P}) \ot (x_{Q} +
(-1)^{j}cx_{Q}).
$$

\subsection{The Drinfeld double of $E(n)$}

Composing the two Hopf algebra isomorphisms $E(n) \to E(n)^{*}$,
$c \mapsto 1^{*} - c^{*}$, $x_{i} \mapsto x_{i}^{*} +
(cx_{i})^{*}$ and $E(n)^{\text{cop}} \to E(n)$, $c \mapsto c$,
$x_{i} \mapsto cx_{i}$, we obtain the isomorphism $E(n) \to
E(n)^{*\text{cop}}$, $c \mapsto 1^{*} - c^{*}$, $x_{i} \mapsto
x_{i}^{*} - (cx_{i})^{*}$. Thus, the Drinfeld double, $D(E(n))$,
is generated by two copies of $E(n)$. Let $C = 1^{*} - c^{*}$ and
$X_{i} = x_{i}^{*} - (cx_{i})^{*}$, $i = 1, \dots, n$. Then,
viewing $E(n)$ and $E(n)^{* \text{cop}}$ as Hopf subalgebras of $D
(E(n))$ and taking into account (\ref{multiplication in D(H)}), we
see that $D (E(n))$ is generated by the grouplike elements $c$ and
$C$, the $(1, c)$-primitive elements $x_{1}, \dots, x_{n}$ and the
$(1, C)$-primitive elements $X_{1}, \dots, X_{n}$, subject to the
following relations:
\begin{gather}
\label{DEn1}
c^{2} = 1, \,\,  x_{i}^{2} = 0, \,\, x_{i}c + cx_{i} = 0, \,\,
x_{i} x_{j} + x_{j} x_{i} = 0, \\
\label{DEn2}
C^{2} = 1, \,\,  X_{i}^{2} = 0, \,\, X_{i}C + CX_{i} = 0, \,\,
X_{i} X_{j} + X_{j} X_{i} = 0,\\
\label{DEn3}
cC = Cc, \,\,  X_{i}c + cX_{i} = x_{i}C + Cx_{i} = 0, \,\, x_{i}
X_{j} + X_{j} x_{i} = \delta_{i,j} (1 - Cc),
\end{gather}
for all $i, j = 1, \dots, n$, where $\delta_{i,j}$ is Kronecker's
delta.


\begin{lemma}
If $P$ is a subset of $\{1, \dots, n\}$ then
\begin{eqnarray*}
x_{P}^{*} & = (-1)^{\frac{|P| (|P| - 1)}{2}} (X_{P} + CX_{P}) \\
(cx_{P})^{*} & = (-1)^{\frac{|P| (|P| + 1)}{2}} (X_{P} - CX_{P}).
\end{eqnarray*}
\end{lemma}
\begin{proof}
For $P \subseteq
\{1, \dots, n\}$ define $Y_{P} = x_{P}^{*} + (cx_{P})^{*}$. An
easy argument using induction on $|P|$ shows that, if $P =
\{i_{1}, \dots, i_{r}\}$, with $i_{1} < i_{2} < \cdots < i_{r}$,
then $Y_{P} = Y_{i_{1}} Y_{i_{2}} \cdots Y_{i_{r}}$. Moreover,
since
$$
(1^{*} - c^{*}) ( x_{P}^{*} + (cx_{P})^{*} ) (c^{i}x_{Q}) =
\left\{ \begin{array}{ccc} 0 & \textnormal{if} & Q \neq P \\
(-1)^{i} & \textnormal{if} & Q = P
\end{array} \right.
$$
we have $CY_{P} = x_{P}^{*} - (cx_{P})^{*}$. In particular,
$CY_{i} = X_{i}$, for all $i = 1, \dots, n$, and, because $C$ is
an element of order 2 that anti-commutes with $X_{i}$, we also
have $Y_{i}C = - CY_{i}$, for all $i$. Consider now $i \in \{0,
1\}$ and $P = \{i_{1}, \dots, i_{r}\}$, with $i_{1} < i_{2} <
\cdots < i_{r}$. Then
\begin{align*}
C^{i}X_{P} & = C^{i} X_{i_{1}} X_{i_{2}} \cdots X_{i_{r}} \\
& = C^{i} (CY_{i_{1}}) (CY_{i_{2}}) \cdots (CY_{i_{r}}) \\
& = (-1)^{\frac{r (r-1)}{2}} C^{r + i} Y_{i_{1}} Y_{i_{2}} \cdots
Y_{i_{r}}\\
& = (-1)^{\frac{|P| (|P| - 1)}{2}} C^{|P| + i} Y_{P}\\
& = (-1)^{\frac{|P| (|P| - 1)}{2}} (x_{P}^{*} + (-1)^{|P| + i}
(cx_{P})^{*}).
\end{align*}
From this we easily obtain the formulas for $x_{P}^{*}$ and
$(cx_{P})^{*}$.
\end{proof}

\begin{remark} For $i \in \{0,
1\}$ and $P \subseteq \{1, \dots, n\}$ we have
\begin{equation} \label{change of base}
(c^{i}x_{P})^{*} = \frac{1}{2} (-1)^{\frac{|P| (|P| - 1)}{2} +
i|P|} (X_{P} + (-1)^{i}CX_{P})
\end{equation}
\end{remark}

\begin{lemma}
\label{chi}
The algebra $D(E(n))$ has a
unique non-trivial one-dimensional representation, $\chi : D(E(n))
\to k$, defined by
$$
\chi (C) = \chi(c) = -1, \quad \chi(x_{i}) = \chi(X_{i}) = 0,
\quad i = 1, \dots, n
$$
\end{lemma}
\begin{proof}
It follows from relations \eqref{DEn1}-\eqref{DEn3} that for
a one-dimensional  representation $\chi: D(E(n))\to k$ one has
$\chi(X_i)=\chi(x_i)=0$ for all $i=1,\dots,n$, $\chi(c)^2=\chi(C)^2=1$,
and $\chi(cC)=1$. This implies the claim.
\end{proof}

\subsection{The Picard group of $\C_n$}
\label{Picard Cn section}

The group $\Pic(\C_n)$ corresponding to the symmetric braiding of $\C_n$ (or, equivalently, the Brauer
group of the Hopf algebra $E(n)$ with quasi-triangular structure $R_0$ (see Section~\ref{En section}))
was computed by Carnovale and Cuadra in \cite{CC}.  This group is isomorphic to $\Sym_n(k)\times \mathbb{Z}/2\mathbb{Z}$.
The torsion subgroup  of $\Pic(\C_n)$  consists of elements induced from $\sVec \subset \C_n$ (see Section~\ref{induction section}).

Let us describe the connected component  of the identity $\Pic_0(\C_n)\subset \Pic(\C_n)$, i.e.,
\[
\Pic_0(\C_n) :=  \Sym_n(k)
\]
in a way suitable for our purposes.   For this end, we identify
$\Pic(\C_n)$  with  the group $\Aut^{br}(\Z(\C_n); \C_n)$ via
Theorem~\ref{image of Pic}.  Note that each $\sigma\in
\Hinv(E(n))$ satisfies condition \eqref{R-invariance of sigma}
and so by Proposition~\ref{Hinv to Pic} the assignment
\begin{equation}
\label{Gamma-sigma}
\Hinv(E(n)) \to \Aut^{br} \big (\Z(\C_n) \big): \sigma \mapsto \Gamma_{F_\sigma},
\end{equation}
where $\Gamma_{F_\sigma}$ is defined in Example~\ref{twisted
YD-modules2}, takes values in $\Aut^{br}(\Z(\C_n); \C_n)$.   It
follows from \cite[Theorem 5.3]{CC}  that \eqref{Gamma-sigma}
restricts to an isomorphism between $\Hinv(E(n))$ and
$\Pic_0(\C_n)$.

\section{The categorical Lagrangian Grassmannian of $\C_n$}

\subsection{Subcategories of $\Z(\C_n)$}

\begin{lemma}
\label{surjective maps D(E(n)) --> E(n)}
The set of surjective Hopf algebra maps $D (E(n)) \to E(n)$ is in
bijection with the set of $n \times 2n$ matrices of rank $n$. The
homomorphism $f$ corresponding to $(A | B) \in M_{n \times 2n} (k)$, where
$A = (a_{ij})$ and $B = (b_{ij})$ are $n \times n$ matrices, is
given by
\begin{equation} \label{surjective maps}
f (C) = f (c) = c, \quad f (X_{i}) = \sum_{j = 1}^{n} a_{ji}x_{j},
\quad f (x_{i}) = \sum_{j = 1}^{n} b_{ji}x_{j}, \quad i = 1,
\dots, n
\end{equation}
\end{lemma}

\begin{proof}
Let $f : D (E (n)) \to E (n)$ be a Hopf algebra map. Since $C$ and
$c$ are group-like elements of $D (E(n))$, we have $f(C)$, $f(c)
\in G(E(n)) = \{1, c\}$. If $(f (C), f(c)) = (1,1)$ then
$f(X_{i})$ and $f (x_{i})$, $i = 1, \dots, n$, are primitive
elements of $E(n)$, so $f(X_{i}) = f(x_{i}) = 0$, for all $i = 1,
\dots, n$. Thus, $f$ is the trivial homomorphism, $f(h) =
\varepsilon (h)1$, which is not surjective. If $(f (C), f(c)) =
(1, c)$ then $f (X_{i}) = 0$, for all $i = 1, \dots, n$. Applying
$f$ to the relation $x_{i}X_{i} + X_{i}x_{i} = 1 - Cc$ we obtain
$0 = 1 - c$, which is not possible. Similarly, if $(f (C), f(c)) =
(c, 1)$. If $(f(C), f(c)) = (c, c)$ then $f (X_{i})$ and $f
(x_{i})$ are $(1, c)$-primitive elements of $E(n)$, for all $i =
1, \dots, n$. Since  the space of $(1,\,c)$-primitive elements of $E(n)$ is $k (1-c) \oplus k
x_{1} \oplus \cdots \oplus kx_{n}$ it follows that there exist
$a$, $b$, $a_{ij}$, $b_{ij} \in k$, $i, j = 1, \dots, n$, such
that $f (X_{i}) = a (1 - c) + \sum_{j} a_{ji}x_{j}$ and $f (x_{i})
= b (1 - c) + \sum_{j}b_{ji}x_{j}$, for all $i = 1, \dots, n$.
Using the relations $x_{i} c + cx_{i} = 0$ and $X_{i}C + CX_{i} =
0$, we readily deduce that $a = b = 0$. Since the remaining
relations impose no other restrictions on the scalars $a_{ij}$ and
$b_{ij}$ we are left to see under what conditions is the
homomorphism associated to these scalars surjective. We claim that
$f$ is surjective if and only if $f$ maps $U = \text{span}
\{X_{1}, \dots, X_{n}, x_{1}, \dots, x_{n} \}$ onto $\text{span}
\{x_{1}, \dots, x_{n} \}$. For this, it suffices to prove that if
$x_{i}$ is in the image of $f$ than it is in the image of the
restriction of $f$ to $U$. Suppose $x_{i} = f(h)$, for some $h \in
D(E(n))$. Since $B = \{C^{j}X_{P}c^{l}x_{Q} \mid j, l \in \{0,
1\}, P, Q \subseteq \{1, \dots, n\}\}$ is a basis of $D(E(n))$
there exist $u \in U$, $v \in V = \text{span} \{ CX_{j}c, Ccx_{j}
\mid j = 1, \dots, n\} $ and $w \in W = \text{span} \, B \setminus
\{X_{j}, x_{j}, CX_{j}c, Ccx_{j} \mid j = 1, \dots, n\}$ such that
$h = u + v + w$. Now $f (u), f(v) \in \text{span} \{x_{1}, \dots,
x_{n}\}$ and $f(w) \in \text{span} \{c^{j}x_{P}\} \setminus
\{x_{1}, \dots, x_{n}\}$, so, from $x_{i} = f(u) + f(v) + f(w)$ we
deduce that $f(w) = 0$. Taking into account that $f (CX_{j}c) = -
f(X_{j})$ and $f(Ccx_{j}) = f(x_{j})$, for all $j = 1, \dots, n$,
we see that $f (v) \in f(U)$, hence $x_{i} \in f(U)$. Thus, $f$ is
surjective if and only if $f$ maps $U$ onto $\text{span} \{x_{1},
\dots, x_{n} \}$. In terms of the scalars $a_{ij}$ and $b_{ij}$
this is equivalent to saying that the rank of the $n \times 2n$
matrix $(A | B)$, where $A = (a_{ij})$ and $B = (b_{ij})$, is $n$.
This proves the lemma.
\end{proof}

\begin{proposition}
\label{L(Cn)} The set $\mathbb{L} (\C_n)$ of subcategories of
$\Z(\C_n)$ equivalent to $\C_n$ as a tensor category is identified
with $\textnormal{Gr} \,(n ,2n)$, the Grassmannian of
$n$-dimensional subspaces of a $2n$-dimensional vector space.
\end{proposition}
\begin{proof}
By virtue of Proposition \ref{subcategories of Rep(H)}, we
identify $\mathbb{L} (\C_n)$ with the set of equivalence classes
of surjective Hopf algebra maps $D(E(n)) \to E(n)$ under the
equivalence relation given by composition with automorphisms of
$E(n)$. Using the description of Lemma \ref{surjective maps
D(E(n)) --> E(n)} this equivalence relation corresponds to the
equivalence relation on $M_{n \times 2n} (k)$ given by left
multiplication with invertible $n$-by-$n$ matrices.
The quotient set associated to the latter is $\Gr (n, 2n)$.
Indeed, if $A = (a_{ij})$ and $B = (b_{ij})$ are two $n \times 2n$
matrices of rank $n$ then the rows of $A$, $r_{1} (A), \dots,
r_{n}(A)$, and the rows of $B$, $r_{1} (B), \dots, r_{n} (B)$,
generate the same subspace of $k^{2n}$ if and only if there exists
$T = (t_{ij}) \in \GL_{n} (k)$ such that $r_{i} (B) = \sum_{j}
t_{ij} r_{j} (A) = r_{i} (TA)$, for all $i = 1, \dots, n$, that
is, if and only if $B = TA$.
\end{proof}

\subsection{$\mathbb{L}_0(\C_n)$ as the Lagrangian Grassmannian of a symplectic form}

\begin{lemma}
If $f : D (E(n)) \to E(n)$ is given by $(\ref{surjective maps})$
and $P = \{i_{1}, \dots, i_{r}\}$ is a subset of $\{1, \dots, n\}$
such that $i_{1} < i_{2} <  \cdots < i_{r}$, then
\begin{equation} \label{f(X) and f(x)}
f (X_{P}) = \sum_{|F| = |P|} [A]_{F, P} x_{F} \qquad and \qquad
f(x_{P}) = \sum_{|F| = |P|} [B]_{F, P} x_{F}
\end{equation}
\end{lemma}
\begin{proof}
We have
\begin{align*}
f(X_{P}) & = f(X_{i_{1}}) \cdots f(X_{i_{r}}) \\
& = \sum_{j_{1}, \dots, j_{r}} a_{j_{1}, i_{1}} \cdots a_{j_{r},
i_{r}} x_{j_{1}} \cdots x_{j_{r}} \\
& = \sum_{\substack{j_{1} < \cdots < j_{r}\\ \sigma \in S_{r}}}
a_{\sigma (j_{1}), i_{1}} \cdots a_{\sigma(j_{r}),
i_{r}} x_{\sigma(j_{1})} \cdots x_{\sigma (j_{r})}\\
& = \sum_{j_{1} < \cdots < j_{r}} \left( \sum_{\sigma \in S_{r}}
\textnormal{sgn} (\sigma) a_{\sigma (j_{1}), i_{1}} \cdots
a_{\sigma(j_{r}), i_{r}} \right) x_{j_{1}}
\cdots x_{j_{r}}\\
& = \sum_{|F| = |P|} [A]_{F, P} x_{F}
\end{align*}
and similarly for $f(x_{P})$.
\end{proof}

Let  $\omega:  k^{2n} \times k^{2n} \to k $ be the symplectic
bilinear form:
\begin{equation}
\omega(\mathbf{a},\, \mathbf{b}) =  \sum_{i = 1}^{n} ( a_{i} b_{n
+ i} - a_{n + i} b_{i}), \qquad \mathbf{a} = (a_{1}, \dots,
a_{2n}),\,\mathbf{b} = (b_{1}, \dots, b_{2n}),
\end{equation}
and let
\begin{eqnarray}
Sp_{2n}(k) &=& \{T\in GL_{2n}(k) \mid \omega(T(\mathbf{a}),\,
T(\mathbf{b}))  = \omega(\mathbf{a},\, \mathbf{b})   \}, \\
PSp_{2n}(k) &=& Sp_{2n}(k) / \{ \pm I_{2n}\}.
\end{eqnarray}
be the symplectic group  and the projective symplectic group, respectively.
Recall that a subspace $V$ of $k^{2n}$ is called {\em isotropic} if $\omega
(\mathbf{a}, \mathbf{b}) = 0$, for all $\mathbf{a}$, $\mathbf{b}
\in V$. An isotropic subspace $V$ is called {\em Lagrangian} if $\dim_k(V)=n$
(which is the maximal possible dimension of an isotropic subspace).

Recall that for a braided tensor category $\C$ we denote by
$\mathbb{L}_{0} (\C)$ the set of tensor subcategories of $\Z(\C)$
braided equivalent to $\C$.

\begin{proposition}
\label{Lagrangian Grassmannian}
$\mathbb{L}_{0} (\C_{n}) = \textnormal{Lag}(n, 2n)$, the
Grassmanian of Lagrangian subspaces of the symplectic space
$(k^{2n},\,\omega)$.
\end{proposition}

\begin{proof}
Under the correspondence of Proposition \ref{subcategories of
Rep(H)}, $\mathbb{L}_{0} (\C_{n})$ is identified with the set of
equivalence classes of surjective Hopf algebra maps $D(E(n)) \to
E(n)$ that take the canonical quasi-triangular structure of
$D(E(n))$ to a triangular structure.

Let $A$, $B \in M_{n}(k)$ be such that the two block matrix $M =
(A | B)$ has rank $n$ and let $f : D(E(n)) \to E(n)$ be the map
given by (\ref{surjective maps}). Let $R = \sum_{i, P} c^{i}x_{P}
\ot (c^{i}x_{P})^{*}$ be the canonical $R$-matrix of $D (E(n))$.
Then, taking into account (\ref{f(X) and f(x)}) and using
(\ref{change of base}), we have:
\begin{align*}
(f \ot f) (R) & = \frac{1}{2} \sum_{i, P} (-1)^{\frac{|P| (|P|
-1)}{2} + i |P|} f (c^{i}x_{P}) \ot f (X_{P} + (-1)^{i} CX_{P}) \\
& = \frac{1}{2} \sum_{i, |E| = |F| = |P|} (-1)^{\frac{|P| (|P|
-1)}{2} + i |P|} [A]_{E, P} [B]_{F, P} c^{i}x_{E} \ot (x_{F} +
(-1)^{i} cx_{F})\\
& = \frac{1}{2} \sum_{|E| = |F| = |P|} (-1)^{\frac{|P| (|P|
-1)}{2}} [A]_{E, P} [B]_{F, P} \Big( x_{E} \ot x_{F} + x_{E} \ot
cx_{F} + \\
& \quad + (-1)^{|P|} cx_{E} \ot x_{F} + (-1)^{|P| + 1} cx_{E} \ot
cx_{F} \Big)\\
& = \frac{1}{2} \sum_{j = 0}^{n} (-1)^{\frac{j (j - 1)}{2}}
\sum_{|E| = |F| = |P| = j} [A]_{E, P} [B]_{F, P} \Big( x_{E} \ot
x_{F}
+ x_{E} \ot cx_{F} + \\
& \quad + (-1)^{j} cx_{E} \ot x_{F} + (-1)^{j + 1} cx_{E} \ot
cx_{F} \Big)\\
& = \frac{1}{2} \sum_{j = 0}^{n} (-1)^{\frac{j (j - 1)}{2}}
\sum_{|E| = |F| = j} \left( \sum_{|P| = j} [A]_{E, P} [B]_{F, P}
\right) \Big(x_{E} \ot x_{F} +  \\
& \quad + x_{E} \ot cx_{F} + (-1)^{j} cx_{E} \ot x_{F} + (-1)^{j +
1} cx_{E} \ot cx_{F} \Big)\\
& = \frac{1}{2} \sum_{j = 0}^{n} (-1)^{\frac{j (j - 1)}{2}}
\sum_{|E| = |F| = j} [AB^{t}]_{E, F} \Big(x_{E} \ot x_{F} + x_{E}
\ot cx_{F} +  \\
& \quad +  (-1)^{j} cx_{E} \ot x_{F} + (-1)^{j +
1} cx_{E} \ot cx_{F} \Big)\\
& = R_{AB^{t}}
\end{align*}
where $B^{t}$ denotes the transpose matrix of $B$ and where we
used the well known formula for the minor of a product of two
matrices, $[AB]_{E, F} = \sum_{|P| = |E|} [A]_{E, P} [B]_{P, F}$.
Thus, $f$ takes the canonical $R$-matrix of $D(E(n))$ to the
$R$-matrix corresponding to $AB^{t}$. Recall that the latter is a
triangular structure if and only if $AB^{t}$ is symmetric. This is
equivalent to $AB^{t} = B A^{t}$, or, what is the same, to
$\sum_{l = 1}^{n} a_{il} b_{jl} = \sum_{l = 1}^{n} b_{il} a_{jl}$,
for all $i$, $j = 1, \dots, n$. Subtracting the right hand term in
the previous equality from the other, we obtain $\sum_{l = 1}^{n}
(a_{il} b_{jl} - b_{il} a_{jl}) = 0$, for all $i$, $j = 1, \dots,
n$. In terms of the matrix $M$, if we denote its rows by $r_{1}
(M), \dots, r_{n} (M)$, this condition is equivalent to saying
that $\omega (r_{i}(M), r_{j}(M)) = 0$, for all $i$, $j = 1,
\dots, n$.

We conclude therefore that the surjective Hopf algebra maps
$D(E(n)) \to E(n)$ which take the canonical quasi-triangular
structure of $D(E(n))$ to a triangular structure of $E(n)$
correspond to $n \times 2n$ matrices, of rank $n$, with entries
from $k$, such that the symplectic form $\omega$ on $k^{2n}$
vanishes on the subspace generated by their rows. Equivalence
classes of such maps have, as their correspondent in $\Gr(n, 2n)$,
those subspaces on which the symplectic form vanishes, whence the
assertion in the statement.
\end{proof}

\section{Computation of $\Aut^{br}(\Z(\C_n))$}


\subsection{$\Ext_{\Z(\C_n)}^{1}(\chi, \varepsilon)$ as a symplectic space}

Recall that $\Z(\C_n)=\Rep(D(E(n))$ has precisely two invertible objects: the trivial representation
$\varepsilon$ and the one-dimensional representation $\chi$ from  Lemma~\ref{chi}.

\begin{proposition} \label{the ext space}
The space $\Ext_{\Z(\C_n)}^{1}(\chi, \varepsilon)$ of equivalence
classes of extensions of the one-dimensional representation $\chi$
by the trivial representation $\varepsilon$ is isomorphic to
$k^{2n}$. The equivalence class corresponding to $\mathbf{a} =
(a_{1}, \dots, a_{2n}) \in k^{2n}$ is the one associated to the
extension
$$
0 \to \varepsilon \xrightarrow{i} V_{\mathbf{a}}
\xrightarrow{p} \chi \to 0
$$
where $V_{\mathbf{a}} = k^{2}$ is the $2$-dimensional $D
(E(n))$-module with basis $\{v_{1} = (1, 0), v_{2} = (0, 1)\}$, $D
(E(n))$-action given in matrix form by
\begin{equation} \label{D(E(n))-module structure}
C, c \mapsto \left( \begin{array}{cc} 1 & 0 \\ 0 & -1\end{array}
\right), \,\,  X_{i} \mapsto \left( \begin{array}{cc} 0 & a_{i} \\
0 & 0 \end{array} \right), \,\, x_{i} \mapsto \left(
\begin{array}{cc} 0 & a_{n + i} \\ 0 & 0\end{array} \right), \,\, i
= 1, \dots, n
\end{equation}
and the maps $i$ and $p$ are such that $i(1) = v_{1}$ and
$p(v_{2}) = 1$.
\end{proposition}
\begin{proof}
Let $V$ be an extension of $\chi$ by $\varepsilon$. Then
$V$ comes equipped with two maps $i$ and $p$ such that
$$
0 \to \varepsilon \xrightarrow{i} V \xrightarrow{p} \chi
\to 0
$$
is an exact sequence. Let $v_{1} = i(1)$ and choose $v_{2} \in V$
such that $p(v_{2}) = 1$. Then $\{v_{1}, v_{2}\}$ is a $k$-basis
of $V$ on which the elements of $D(E(n))$ act by $h \cdot v_{1} =
\varepsilon(h) v_{1}$ and $h \cdot v_{2} = f(h) v_{1} + \chi(h)
v_{2}$, for all $h \in D(E(n))$, and for some linear map $f \in
\Hom (D(E(n)), k)$. Consider now another extension
$$
0 \to \varepsilon \xrightarrow{i'} V' \xrightarrow{p'}  \chi \to 0
$$
of $\chi$ by $\varepsilon$ and associate to $V'$, as above, a
basis $\{v'_{1}, v'_{2}\}$ and a linear map $f' \in \Hom (D(E(n)),
k)$. We claim that if there exists a homomorphism of extensions
$\varphi : V \to V'$ then $f$ and $f'$ differ by a multiple of
$\chi - \varepsilon$. Indeed, if $\varphi$ is such a map then,
from $\varphi \circ i = i'$ and $p' \circ \varphi = p$, we readily
deduce that $\varphi (v_{1}) = v'_{1}$ and $\varphi (v_{2}) =
\lambda v'_{1} + v'_{2}$, for some $\lambda \in k$. Letting $h \in
D(E(n))$ act on the latter relation and taking into account that
$\varphi$ commutes with the action of $D(E(n))$, we arrive at the
equality $\big( f(h) + \lambda \chi(h) \big) v'_{1} + \chi(h)
v'_{2} = \big( \lambda \varepsilon(h) + f'(h) \big) v'_{1} +
\chi(h) v'_{2}$, which shows that $f' - f = \lambda (\chi -
\varepsilon)$. In particular, if we take $V' = V$ and $\varphi =
\id_{V}$, we see that the $2n$-tuple $(f(X_{1}), \dots, f(X_{n}),
f(x_{1}), \dots f(x_{n}))$ does not depend on the choice of
$v_{2}$. Also, the above discussion shows that the same $2n$-tuple
depends only on the equivalence class of $V$. We can, thus, define
a map $\Ext_{\Z(\C_n)}^{1} (\chi, \varepsilon) \to k^{2n}$ sending
the equivalence class of $V$ to $(f(X_{1}), \dots, f(X_{n}),
f(x_{1}), \dots f(x_{n}))$. This map is easily seen to be
one-to-one and onto, sending the equivalence class of the
extension $V_{\mathbf{a}}$, in the statement, to $\mathbf{a} \in
k^{2n}$.
\end{proof}

\begin{lemma}
\label{V_a as a Yetter-Drinfeld module}
Let
\[
0 \to \varepsilon \xrightarrow{i} V_{\mathbf{a}} \xrightarrow{p} \chi \to 0
\]
be the extension of $\chi$ by $\varepsilon$ associated to the $2n$-tuple $\mathbf{a} = (a_{1},
\dots, a_{2n}) \in k^{2n}$ and let $\{v_{1}, v_{2}\}$ be a basis of
$V_{\mathbf{a}}$ such that the action of $D(E(n))$ on
$V_{\mathbf{a}}$ is given by $(\ref{D(E(n))-module structure})$.
Then, as a Yetter-Drinfeld module over $E(n)$, $V_{\mathbf{a}}$
has the action  given by the restriction of action of $D(E(n))$ to the copy of
$E(n)$ generated by $c$ and $\{x_{i}\}$, and coaction given by
$$
\rho(v_{1}) = v_{1} \ot 1 \quad and \quad \rho(v_{2}) = \sum_{j =
1}^{n} a_{j} v_{1} \ot x_{j} + v_{2} \ot c.
$$
\end{lemma}

\begin{proof}
We only need to verify the formulas for the coaction. This is
given by
$$
\rho (v) = \sum_{i, P} (c^{i}x_{P})^* \cdot v \ot c^{i}x_{P}
$$
for all $v \in V_{\mathbf{a}}$, so, taking into account
(\ref{change of base}) and the following easy to check relations
$$
(C^{i}X_{P}) \cdot v_{1} = \left\{\begin{array}{lll}
v_{1} & \textnormal{if} & P = \emptyset \\
0     & \textnormal{if} & P \neq \emptyset \\
\end{array}\right.
\quad \textnormal{and} \quad (C^{i}X_{P}) \cdot v_{2} =
\left\{\begin{array}{lll}
(-1)^{i}v_{2} & \textnormal{if} & P = \emptyset \\
a_{j}v_{1}    & \textnormal{if} & P = \{j\} \\
0             & \textnormal{if} & |P| \geq 2
\end{array}\right.
$$
the verification is straightforward.
\end{proof}

\begin{lemma}
\label{Centralizing chi}
Let $V_{\mathbf{a}}$ be the underlying object of an extension of $\chi$ by $\varepsilon$. Then  $V_{\mathbf{a}}$
centralizes $\chi$, i.e., the squared braiding $c_{V_{\mathbf{a}},\chi} \circ c_{\chi,V_{\mathbf{a}}}$ is the identity
morphism.
\end{lemma}
\begin{proof}
As a Yetter-Drinfeld module over $E(n)$, $\chi$ has the module
structure given by $c \cdot 1 = -1$, $x_{i} \cdot 1 = 0$, for all
$i = 1, \dots, n$, and the comodule structure $1 \mapsto 1 \ot c$.
Let $\{v_{1}, v_{2}\}$ be a basis of $V_{\mathbf{a}}$ such that
the action of $D(E(n))$ on $V_{\mathbf{a}}$ is given by
(\ref{D(E(n))-module structure}). Then, from Lemma \ref{V_a as a
Yetter-Drinfeld module} and (\ref{central structure}), we deduce
that
$$
c_{V_{\mathbf{a}},\chi} \circ c_{\chi,V_{\mathbf{a}}} (1 \ot
v_{1}) = c_{V_{\mathbf{a}},\chi} (v_{1} \ot 1) = 1 \ot c \cdot
v_{1} = 1 \ot v_{1}
$$
and
\begin{align*}
c_{V_{\mathbf{a}},\chi} \circ c_{\chi,V_{\mathbf{a}}} (1 \ot
v_{2}) & = c_{V_{\mathbf{a}},\chi} \left( \sum_{j = 1}^{n} a_{j}
v_{1} \ot x_{j} \cdot 1 + v_{2} \ot c \cdot 1 \right)\\
& = - c_{V_{\mathbf{a}},\chi} (v_{2} \ot 1) \\
& = - 1 \ot c \cdot v_{2}\\
& = 1 \ot v_{2}
\end{align*}
whence the claim.
\end{proof}

\begin{corollary}
\label{tau on Ext}
Let $\tau\in \Aut^{br}(\Z(\C_n))$ be the image of the non-trivial element of $\Pic(\sVec)\cong \mathbb{Z}/2\mathbb{Z}$
under the composition
\[
\Pic(\sVec) \xrightarrow{\Ind_{\sVec}^{\Z(\C_n)}} \Pic \big(
\Z(\C_n) \big) \cong  \Aut^{br}\big( \Z(\C_n) \big).
\]
Then $\tau(V_{\mathbf{a}})\cong V_{\mathbf{a}}$.
\end{corollary}
\begin{proof}
This follows from Proposition~\ref{image of induced} and Lemma~\ref{Centralizing chi}.
\end{proof}

\begin{proposition}
\label{squared braiding} Let $V_{\mathbf{a}}$ and $V_{\mathbf{b}}$
be two extensions of $\chi$ by $\varepsilon$ associated to
$\mathbf{a} = (a_{1}, \dots, a_{2n})$ and $\mathbf{b} = (b_{1},
\dots, b_{2n})$ and suppose that $\{v_{1}, v_{2}\}$ and $\{v'_{1},
v'_{2}\}$ are bases of $V_{\mathbf{a}}$ and $V_{\mathbf{b}}$,
respectively, on which $D(E(n))$ acts as in $(\ref{D(E(n))-module
structure})$. Then the matrix of $c_{V_{\mathbf{b}},
V_{\mathbf{a}}} \circ c_{V_{\mathbf{a}}, V_{\mathbf{b}}} :
V_{\mathbf{a}} \ot V_{\mathbf{b}} \to V_{\mathbf{b}} \ot
V_{\mathbf{a}}$ in the basis $\{ v_1\ot v_1',\, v_1\ot v_2',\, v_2\ot v_1',\, v_2\ot v_2'\}$ is
$$
\left( \begin{array}{cccc}
1 & 0 & 0 & \omega(\mathbf{b}, \mathbf{a})\\
0 & 1 & 0 & 0 \\
0 & 0 & 1 & 0\\
0 & 0 & 0 & 1\\
\end{array} \right).
$$
\end{proposition}

\begin{proof} This follows
immediately from (\ref{central structure}) and the Yetter-Drinfeld
module structures of $V_{\mathbf{a}}$ and $V_{\mathbf{b}}$. For
example,
\begin{align*}
c_{V_{\mathbf{b}}, V_{\mathbf{a}}} \circ c_{V_{\mathbf{a}},
V_{\mathbf{b}}} (v_{2} \ot v'_{2}) & = c_{V_{\mathbf{b}},
V_{\mathbf{a}}} \left( \sum_{i = 1}^{n} b_{i} v'_{1} \ot x_{i}
\cdot v_{2} + v'_{2} \ot c \cdot v_{2} \right) \\
& = c_{V_{\mathbf{b}}, V_{\mathbf{a}}} \left( \sum_{i = 1}^{n}
b_{i} a_{n + i} v'_{1} \ot v_{1} - v'_{2} \ot v_{2} \right) \\
& =  \sum_{i = 1}^{n} b_{i} a_{n + i} v_{1} \ot v'_{1} - \sum_{i =
1}^{n} a_{i} v_{1} \ot x_{i} \cdot v'_{2} - v_{2} \ot c \cdot
v'_{2} \\
& =  \sum_{i = 1}^{n} b_{i} a_{n + i} v_{1} \ot v'_{1} - \sum_{i =
1}^{n} b_{n + i}  a_{i} v_{1} \ot v'_{1} + v_{2} \ot v'_{2} \\
& = \omega (\mathbf{b}, \mathbf{a}) v_{1} \ot v'_{1} + v_{2} \ot v'_{2}.
\end{align*}
\end{proof}

For a tensor  subcategory $\mathcal{S}\subset \Z(\C_n)$ let
$\mathcal{S}\cap \Ext_{\Z(\C_n)}^{1}(\chi, \varepsilon)$ denote
the subspace consisting of  equivalence classes of all extensions
$0\to \eps\to V \to \chi\to 0$ such that $V$ belongs to
$\mathcal{S}$.

\begin{proposition}
\label{Grassmannian bijection} The assignment $\L \to \L \cap
\Ext_{\Z(\C_n)}^{1}(\chi, \varepsilon)$ is a bijection between
$\mathbb{L}_0(\C_n)$ and $\text{Lag}(n,\,2n)$.
\end{proposition}

\begin{proof}
We saw in Proposition \ref{Lagrangian Grassmannian} that
$\mathbb{L}_{0} (\C_{n}) = \text{Lag}(n,\,2n)$. Let $U \in
\text{Lag}(n,\,2n)$ and $A = (a_{ij})$, $B = (b_{ij})$ be
$n$-by-$n$ matrices such that the rows of $M = (A|B) \in M_{n
\times 2n} (k)$, $r_{1} (M), \dots, r_{n} (M)$, form a basis of
$U$. Then the subcategory of $\Z (\C_{n})$, braided equivalent to
$\C_{n}$, corresponding to $U$ is $\L_{U}$, where $\L_{U}$ is the
image of the restriction  functor associated to $f:
D(E(n)) \to E(n)$, $f (C) = f(c) = c$, $f (X_{i}) = \sum_{j =
1}^{n} a_{ji} x_{j}$, $f (x_{i}) = \sum_{j = 1}^{n} b_{ji} x_{j}$,
$i = 1, \dots, n$. We will show that, under the isomorphism of
Proposition \ref{the ext space}, $\L_{U} \cap
\Ext_{\Z(\C_n)}^{1}(\chi, \varepsilon) = U$, which will prove the
claim.

Let
$$
0 \to \varepsilon \xrightarrow{i} V \xrightarrow{p} \chi \to 0
$$
be an element of $\L_{U} \cap \Ext_{\Z(\C_n)}^{1}(\chi,
\varepsilon)$ and let $\{v_{1}, v_{2}\}$ be a basis of $V$ such
that $v_{1} = i (1)$ and $p (v_{2}) = 1$. If $\{v_{1}^{*},
v_{2}^{*}\}$ is the dual basis of $\{v_{1}, v_{2}\}$ then the
element of $k^{2n}$ corresponding to $V$, under the isomorphism of
Proposition \ref{the ext space}, is
$$
\mathbf{a}_{V} = \big( v_{1}^* (X_{1} \cdot v_{2}), \dots, v_{1}^*
(X_{n} \cdot v_{2}),  v_{1}^* (x_{1} \cdot v_{2}), \dots, v_{1}^*
(X_{1} \cdot v_{2})\big)
$$
Since $v_{1}^{*} (X_{i} \cdot v_{2}) = v_{1}^{*} \big( f(X_{i})
v_{2} \big) = v_{1}^{*} ( \sum_{j} a_{ji} x_{j} v_{2}) = \sum_{j}
a_{ji} v_{1}^{*} (x_{j} v_{2})$ and $v_{1}^{*} (x_{i} \cdot v_{2})
= v_{1}^{*} \big( f(x_{i}) v_{2} \big) = v_{1}^{*} ( \sum_{j}
b_{ji} x_{j} v_{2}) = \sum_{j} b_{ji} v_{1}^{*} (x_{j} v_{2})$,
for all $i = 1, \dots, n$, we deduce that
$$
\mathbf{a}_{V} = \sum_{j = 1}^{n} v_{1}^{*} (x_{j} v_{2}) (a_{j1},
\dots, a_{jn}, b_{j1}, \dots, b_{jn}) = \sum_{j = 1}^{n} v_{1}^{*}
(x_{j} v_{2}) r_{j} (M)
$$
Thus, $\mathbf{a}_{V} \in U$, for all $V \in \L_{U} \cap
\Ext_{\Z(\C_n)}^{1}(\chi, \varepsilon)$. To complete the proof we
need only show that $V_{r_{i}(M)} \in \L_{U}$, for all $i = 1,
\dots, n$. A quick check shows that the representation
$V_{r_{i}(M)}$ is obtained from the following matrix
representation of $E(n)$:
$$
E(n) \to M_{2}(k), \quad c \mapsto \left( \begin{array}{cc} 1 & 0\\
0 & -1 \end{array}\right), \quad x_{j} \mapsto \left(
\begin{array}{cc} 0 & \delta_{i,j} \\ 0 & 0 \end{array}\right), \quad j
= 1, \dots, n
$$
by restriction of scalars via $f$.
\end{proof}

\subsection{Representation of $\Aut^{br}(\Z(\C_n))$ on $\Ext_{\Z(\C_n)}^{1}(\varepsilon,\,\chi)$}

Recall from Proposition~\ref{action on Ext} that there is a group homomorphism
\begin{equation}
\label{rho}
\rho: \Aut^{br}(\Z(\C_n)) \to PGL(\Ext_{\Z(\C_n)}^1(\varepsilon,\,\chi)) = PGL_{2n}(k).
\end{equation}

Note that the projective symplectic group $PSp_{2n}(k) = Sp_{2n}(k)/\{\pm I_{2n}\}$
can be viewed as a subgroup of $PGL_{2n}(k)$.

\begin{proposition}
\label{symplectic image}
The image of the group homomorphism \eqref{rho}
belongs to $PSp_{2n}(k)$.
\end{proposition}
\begin{proof}
Let
\begin{eqnarray*}
& 0 \to \varepsilon \xrightarrow{i} V_{\mathbf{a}} \xrightarrow{p} \chi \to 0 &\\
& 0 \to \varepsilon \xrightarrow{i'} V_{\mathbf{a'}} \xrightarrow{p'} \chi \to 0  &
\end{eqnarray*}
be a pair of elements in $\Ext_{\Z(\C_n)}^1(\chi,\,\eps)$.
By Proposition~\ref{squared braiding} we have
\begin{equation}
\label{cVaVa'}
c_{V_\mathbf{a'}, V_{\mathbf{a}}} \circ c_{V_\mathbf{a}, V_{\mathbf{a'}}} = \id_{V_\mathbf{a}\ot V_{\mathbf{a'}}}
+\omega(\mathbf{a}, \mathbf{a'})\,  (p \ot p')\circ (i \ot i').
\end{equation}
Here $p\ot p'$ is a morphism from $V_\mathbf{a}\ot  V_{\mathbf{a'}}$ to $\chi\ot \chi=\eps$.
For  $\alpha\in \Aut^{br}(\Z(\C_n))$ let  $\alpha(\mathbf{a})\in \Ext_{\Z(\C_n)}^1(\chi,\,\eps))$ be such that
$\alpha( V_{\mathbf{a}})   = V_{\alpha(\mathbf{a})}$. Applying $\alpha$ to both sides of \eqref{cVaVa'}
we obtain
\begin{equation*}
\alpha (c_{V_\mathbf{a'}, V_{\mathbf{a}}} \circ c_{V_\mathbf{a}, V_{\mathbf{a'}}} )=
\id_{V_{\alpha(\mathbf{a})}\ot V_{\alpha(\mathbf{a'})}}
+\omega(\mathbf{a}, \mathbf{a'})\,  \alpha((p \ot p')\circ (i \ot i')).
\end{equation*}
On the other hand,
\begin{eqnarray*}
\lefteqn{ \alpha (c_{V_\mathbf{a'}, V_{\mathbf{a}}} \circ c_{V_\mathbf{a}, V_{\mathbf{a'}}} )} \\
&=& J_{V_\mathbf{a}, V_{\mathbf{a'}}} \circ
( c_{V_{\alpha(\mathbf{a'})}, V_{\alpha({\mathbf{a}})}} \circ c_{V_{\alpha(\mathbf{a})}, V_{\alpha({\mathbf{a'}})}}) \circ
J^{-1}_{V_\mathbf{a}, V_{\mathbf{a'}}}   \\
&=&  \id_{V_{\alpha(\mathbf{a})}\ot V_{\alpha(\mathbf{a'})}}
+\omega(\alpha(\mathbf{a}), \alpha(\mathbf{a'}))\,
J_{V_\mathbf{a}, V_{\mathbf{a'}}} \circ  (\alpha(p)\ot \alpha(p')) \circ (\alpha(i) \ot \alpha(i'))  \circ J^{-1}_{V_\mathbf{a}, V_{\mathbf{a'}}},
\end{eqnarray*}
where $J_{X,Y}: \alpha(X)\ot \alpha(Y) \xrightarrow{\sim} \alpha(X\ot Y)$ denotes the tensor functor structure of $\alpha$.

Thus, $\omega(\alpha(\mathbf{a}), \alpha(\mathbf{a'})) = \omega(\mathbf{a}, \mathbf{a'})$, as required.
\end{proof}

Recall the homomorphisms
\begin{eqnarray*}
\iota_1 &:&  \Aut_{\text{Hopf}}(H)\to \Aut^{br}(\Z(\Rep(H)),\\
\iota_2 &:&  \Hinv(H) \to \Aut^{br}(\Z(\Rep(H)),\\
\iota_3 &:&  \Hinv(H^*) \to \Aut^{br}(\Z(\Rep(H)).
\end{eqnarray*}
from Remark~\ref{iotas}. In the next Proposition  we compute the images
of the compositions of these homomorphisms with \eqref{rho}.

\begin{proposition}
\label{3 images}
We have
\begin{enumerate}
\item[(i)] $\rho\circ \iota_1 \big( \Aut_{\textnormal{Hopf}}(E(n))
\big) = \left\{
\begin{pmatrix} A  & 0 \\ 0 & (A^t)^{-1} \end{pmatrix} \mid A\in GL_{2n}(k) \right\}$,
\item[(ii)] $\rho \circ \iota_2\big(
\textnormal{H}^{2}_{\textnormal{inv}}(E(n)) \big) = \left\{
\begin{pmatrix} I_{n} & 0 \\ B & I_{n} \end{pmatrix} \mid B = B^t \right\}$,
\item[(iii)] $\rho \circ \iota_3
\big(\textnormal{H}^{2}_{\textnormal{inv}}(E(n)^*)\big) = \left\{
\begin{pmatrix} I_{n}  & B \\ 0 & I_{n} \end{pmatrix} \mid B = B^t
\right\}$.
\end{enumerate}
Here each matrix denotes the class in $PSp_{2n}(k)$.
\end{proposition}
\begin{proof}
(i)  This is clear in view of Proposition~\ref{Aut-br Cn}.

(ii) Let $\sigma_{M} \in (E(n) \ot E(n))^{*}$ be the invariant
$2$-cocycle associated to $M = (m_{ij}) \in \Sym_{n}(k)$,
$\Gamma_{\sigma_{M}}$ the autoequivalence of $\Rep \big( D (E(n))
\big)$ induced by $\sigma_{M}$ and $V_{\mathbf{a}} \in \Ext_{\Z(\C_n)}^{1}
(\chi, \varepsilon)$, where $\mathbf{a} = (a_{1}, \dots, a_{2n})
\in k^{2n}$. Viewing $V_{\mathbf{a}}$ as a Yetter-Drinfeld module
over $E (n)$, by Lemma \ref{V_a as a Yetter-Drinfeld module}, and
regarding $\Gamma_{\sigma_{M}}$ as an autoequivalence of $_{E(n)}
\mathcal{YD}^{E(n)}$, we have, according to Example \ref{twisted
YD-modules2}, that $\Gamma_{\sigma_{M}} (V_{\mathbf{a}}) =
V_{\mathbf{a}}$ as an $E(n)$-comodule, with the $E (n)$-module
structure given by
$$
h \cdot v = \sigma_{M}^{-1} \Big( \big( h_{(2)} \cdot z_{(0)}
\big)_{(1)}\ot  h_{(1)} \Big) \sigma_{M} (h_{(3)}\ot  z_{(1)}) \big(
h_{(2)} \cdot z_{(0)} \big)_{(0)}, \quad h \in E(n), v \in
V_{\mathbf{a}}.
$$
Let $\{v_{1}, v_{2} \}$ be a basis of $V_{\mathbf{a}}$ such that
the $D (E(n))$-module structure of $V_{\mathbf{a}}$ is given by
(\ref{D(E(n))-module structure}). Then, a straightforward
computation shows that, $h \cdot v_{1} = \varepsilon(h) v_{1}$,
for all $h \in E(n)$, $c \cdot v_{2} = - v_{2}$ and
$$
x_{i} \cdot v_{2} = a_{n + j} + \left[ \sum_{i = 1}^{n} (m_{ij} +
m_{ji}) a_{i} \right] v_{1}
$$
for all $i = 1, \dots, n$. Thus, $\Gamma_{\sigma_{M}}
(V_{\mathbf{a}}) = V_{\mathbf{a}'}$, where
$$
\mathbf{a}'^{t} = \left( \begin{array}{cc}
I_{n}     & 0 \\
M + M^{t} & I_{n}
\end{array} \right) \mathbf{a}^{t}
$$
and the result follows.

(iii) Let $M = (m_{ij}) \in \Sym_{n}(k)$ and let
$$
J_{M} = \frac{1}{4} \sum_{i, j, P, Q} \sigma_{M} (c^{i}x_{P}\ot
c^{j} x_{Q}) (x_{P} + (-1)^{i} cx_{P}) \ot (x_{Q} +
(-1)^{j}cx_{Q})
$$
be the invariant twist associated to $M$. Observe that
$$
J_{M} = 1 \ot 1 + \sum_{j, l = 1}^{n} m_{jl} x_{j} \ot cx_{l} + L
$$
where $L$ is a linear combination of $c^{i}x_{P} \ot c^{j}x_{Q}$,
with $i, j \in \{0, 1\}$ and $P$, $Q \subseteq \{1, \dots, n\}$
such that $|P| \geq 2$ or $|Q| \geq 2$. Let $V_{\mathbf{a}} \in
\Ext_{\Z(\C_n)}^{1} (\chi,\, \varepsilon)$ and $\Gamma_{J_{M}}$ be the
autoequivalence of $\Rep \big( D (E(n)) \big)$ induced by $J_{M}$.
As a Yetter-Drinfeld module over $E(n)$, $\Gamma_{J_{M}}
(V_{\mathbf{a}})$ is $V_{\mathbf{a}}$ as an $E(n)$-module, with
the comodule structure given by
$$
\rho_{J_{M}} (v) = (J_{M}^{-1})^{2} \cdot (J_{M}^{1} \cdot v)\0
\ot (J_{M}^{-1})^{1}(J_{M}^{1} \cdot v)\1 J_{M}^{2},
$$
for all $v \in V_{\mathbf{a}}$. If $\{v_{1}, v_{2}\}$ is a basis
for $V_{\mathbf{a}}$ such that the action of $D(E(n))$ on
$V_{\mathbf{a}}$ is given by (\ref{D(E(n))-module structure}),
then one can easily check, using Lemma \ref{V_a as a
Yetter-Drinfeld module}, that $\rho_{J_{M}} (v_{1}) = v_{1} \ot 1$
and
$$
\rho_{J_{M}} (v_{2}) = \sum_{i = 1}^{n} a_{i} v_{1} \ot x_{i} +
\sum_{i = 1}^{n} \left( \sum_{j = 1}^{n} a_{n + j} (m_{ij} +
m_{ji}) \right) v_{1} \ot cx_{i} + v_{2} \ot c.
$$
Taking into account that the induced $E(n)^{*}$-module structure
of $\Gamma_{J_{M}} (V_{\mathbf{a}})$ is $f \cdot v = \sum
f(v_{(1)}) v_{(0)}$, for all $f \in E(n)^{*}$ and $v \in
\Gamma_{J_{M}} (V_{\mathbf{a}})$ we readily deduce the
$D(E(n))$-module structure of $\Gamma_{J_{M}} (V_{\mathbf{a}})$.
We have $C \cdot v_{1} = v_{1}$, $C \cdot v_{2} = - v_{2}$, $X_{i}
\cdot v_{1} = 0$ and
$$
X_{i} \cdot v_{2} = (x_{i}^{*} - (cx_{i})^{*}) \cdot v_{2} =
\left( a_{i} - \sum_{j = 1}^{n} a_{n + j} (m_{ij} + m_{ji})
\right) v_{1}.
$$
Thus, $\Gamma_{J_{M}} (V_{\mathbf{a}}) = V_{\mathbf{a}'}$, where
$$
\mathbf{a}'^{t} = \left( \begin{array}{cc}
I_{n} & -(M + M^{t}) \\
0     & I_{n}
\end{array} \right) \mathbf{a}^{t}
$$
which concludes the proof.
\end{proof}

\begin{corollary}
\label{precise image of rho}
The image  of homomorphism \eqref{rho} is $PSp_{2n}(k)$.
\end{corollary}
\begin{proof}
The three subgroups from Proposition~\ref{3 images}
generate $PSp_{2n}(k)$, so the statement follows
from Proposition~\ref{symplectic image}.
\end{proof}

Recall from Corollary~\ref{tau on Ext}
that $\tau \in \Aut^{br}(\Z(\C_n))$ denotes the  autoequivalence induced from the
non-trivial element of $\Pic(\sVec)$, where  $\sVec$ is viewed as a  tensor
subcategory of  $\Z(\C_n)$.

\begin{proposition}
\label{Z/2Z}
\begin{enumerate}
\item[(i)] The kernel of  homomorphism \eqref{rho} is $\langle
\tau \rangle$ (and so  is isomorphic to $\mathbb{Z}/2\mathbb{Z}$).
\item[(ii)] The restriction of homomorphism \eqref{rho}  on the
subgroup of $\Aut^{br}(\Z(\C_n))$ generated by the images of
$\Aut_{\textnormal{Hopf}}(E(n))$,
$\textnormal{H}^{2}_{\textnormal{inv}} (E(n))$,
$\textnormal{H}^{2}_{\textnormal{inv}} (E(n)^*)$ is injective.
\end{enumerate}
\end{proposition}
\begin{proof}
(i) By Theorem~\ref{image of Pic} and Proposition~\ref{Grassmannian bijection}
the kernel of $\rho$ is a subgroup of $\Pic(\C_n)= \Pic_0(\C_n) \times \langle \tau \rangle$.
By Corollary~\ref{tau on Ext} the subgroup  $\langle \tau \rangle \subset \Pic(\C_n)$
acts trivially on the (projective) $\Ext$ space and so belongs to the kernel of $\rho$.

In Section~\ref{Picard Cn section} we discussed an isomorphism between $\Hinv(E(n)^*)$
and $\Pic_0(\C_n)$ (both groups are isomorphic to $\Sym_n(k)$).  Combining this with
Proposition~\ref{3 images}(iii) we deduce that the kernel of $\rho$ coincides with $\langle \tau \rangle$.

(ii) Every matrix $M\in Sp_{2n}(k)$ can be uniquely written as $M=XYZ$, where $X,\,Y,\,Z$ are matrices
from parts (i), (ii), and (iii) of Proposition~\ref{3 images}, respectively.  This implies the injectivity statement.
\end{proof}

\begin{theorem}
\label{main thm}
We have
\begin{equation}
\Aut^{br}(\Z(\C_n)) \simeq PSp_{2n}(k) \times \mathbb{Z}/2\mathbb{Z}.
\end{equation}
The action of $\Aut^{br}(\Z(\C_n))$ on $\mathbb{L}_0(\C_n)$ corresponds to the action
of $PSp_{2n}(k)$ on the Lagrangian Grassmannian $\text{Lag}(n,\, 2n)$.
\end{theorem}
\begin{proof}
That $\mathbb{Z}/2\mathbb{Z} =\langle \tau \rangle$ is a central subgroup of $\Aut^{br}(\Z(\C_n))$
and that the quotient by it is isomorphic to $PSp_{2n}(k)$
follows from Corollary~\ref{precise image of rho} and Proposition~\ref{Z/2Z}(i).
Thus we have a central extension
\begin{equation}
\label{central extension}
1 \to \langle \tau \rangle  \to  \Aut^{br}(\Z(\C_n))  \to PSp_{2n}(k) \to 1.
\end{equation}
This extension splits by Proposition~\ref{Z/2Z}(ii).

Finally, the Lagrangian equivariance follows from Proposition~\ref{Grassmannian bijection}.
\end{proof}

\begin{corollary}
\label{main cor}
$\BrPic(\C_n) \cong PSp_{2n}(k) \times \mathbb{Z}/2\mathbb{Z}$.
\end{corollary}
%


\end{document}